\documentclass[11pt, oneside]{article}

\usepackage[dvipsnames]{xcolor}

\usepackage[totalwidth=480pt,totalheight=680pt]{geometry}              		
\usepackage[british]{babel}
\usepackage{graphicx}				
\usepackage{amsmath}								
\usepackage{amssymb}
\usepackage{amsthm}
\RequirePackage[colorlinks,citecolor=blue,urlcolor=blue]{hyperref}
\usepackage{verbatim}
\usepackage{enumerate}

\allowdisplaybreaks
\numberwithin{equation}{section}

\newcommand{\Expect}{\operatorname{\mathbb{E}}}
 \newcommand{\Indicator}{\operatorname{\mathbb{I}}}
 
 \newcommand{\Ito}{It\^o}
 
 \newcommand{\Prob}{\operatorname{\mathbb{P}}}
 \newcommand{\Reals}{\mathbb{R}}

 \newcommand{\talpha}{\widetilde{\alpha}}
 \newcommand{\tN}{\widetilde{N}}
 \newcommand{\stau}{\tau^H}
 \newcommand{\vtau}{\tau^V}

 \newtheorem{thm}{Theorem}[section]
 \newtheorem{lem}[thm]{Lemma}
 \newtheorem{cor}[thm]{Corollary}

\theoremstyle{definition}
\newtheorem{remark}[thm]{Remark}

\newcommand{\EE}{\Expect}
\newcommand{\PP}{\Prob}
\providecommand{\rpar}[1]{\left( #1 \right)}               
\providecommand{\kpar}[1]{\left\{ #1 \right\}}              
\providecommand{\ppar}[1]{\left\langle #1 \right\rangle}     
\providecommand{\abs}[1]{\left\vert #1 \right\vert}         
\renewcommand{\L}{\mathcal{L}}
\newcommand{\B}{\mathcal{B}}
\newcommand{\C}{\mathcal{C}}
\newcommand{\F}{\mathcal{F}}
\newcommand{\HH}{\mathbb{H}}

\newcommand{\wh}{\widehat}
\newcommand{\A}{A}
\newcommand{\tgam}{T}
\newcommand{\eps}{\varepsilon}
\newcommand{\prt}{\partial}

\usepackage{color} 
\title{Gravitation versus Brownian motion}
\author{Sayan Banerjee \ \ \ Krzysztof Burdzy \ \ \ Mauricio Duarte}

\begin{document}
\maketitle
\begin{abstract}
We investigate the motion of an inert (massive) particle being impinged from below by a particle performing (reflected) Brownian motion. The velocity of the inert particle increases in proportion to the local time of collisions and decreases according to a constant downward gravitational acceleration. We study fluctuations and strong laws of the motion of the particles. We further show that the joint distribution of the velocity of the inert particle and the gap between the two particles converges in total variation distance to a stationary distribution which has an explicit product form.
\end{abstract}

\textbf{Keywords and phrases: }Brownian motion, inert drift, local time, gravitation, total variation distance.\\\\
\textbf{AMS 2010 subject classifications: }Primary 60J65; secondary 60J55.

\section{Introduction}

We will investigate the motion of an inert (massive) particle that is impinged from below by a particle performing (reflected) Brownian motion. 
Whenever the two particles collide, the velocity of the inert particle increases in proportion to the local time of collisions. Furthermore, there is a gravitational field that pulls the inert particle downwards by giving it a constant acceleration. 
The Brownian particle is reflected on the trajectory of the inert particle according to the usual Skorokhod recipe.

Formally, the motion of the two particles will be defined by a system of SDE's.
We will denote the driving Brownian motion by $B$. We will use $X$ and $S$ to denote the trajectories of the reflecting Brownian particle and the inert particle, respectively, and $V$ to denote the velocity of the inert particle. Gravitation will be represented by a constant acceleration $g>0$.
We will write  $L$ to denote the intersection local time between the two particles, defined as the unique continuous non-decreasing process increasing only when $S_t=X_t$ (i.e., $L_t-L_0=\int_0^t\Indicator_{\{S_u=X_u\}}dL_u$ for all $t \ge 0$). The SDE's are
\begin{align}\label{eqnarray:eqmotion}
\begin{cases}
d X_t = d B_t - d L_t, \\
d V_t = d L_t-g d t,\\
d S_t = V_t d t. 
\end{cases}
\end{align}
There is also an extra condition that $S_t \ge X_t$ for all $t \ge 0$, that is, the Brownian particle and the inert particle can collide but their trajectories cannot cross (this applies, in particular, to the initial condition, i.e., $S_0 \geq X_0$). We will show  existence and uniqueness  of the strong solution to 
\eqref{eqnarray:eqmotion} in Theorem \ref{thm:existence_sde}.

The model without the gravitational component was originally introduced in \cite{knight}. The motivation came from trying to mathematically model the joint motion of a Brownian particle in a liquid and a semi-permeable membrane (thought of as the inert particle) which is permeable to the microscopic liquid molecules but not to the macroscopic Brownian particle. Without gravitation, the inert particle moves with constant velocity in the absence of collisions and the velocity increases (in proportion to the local time of collisions) only when the particles collide. Thus, it is clear that there will be a random time after which the particles never collide and the inert particle ``escapes" the Brownian particle with constant velocity. The laws of the inverse velocity process $V^{(-1)}$ and the ``escape velocity" were explicitly computed in \cite{knight}.

The effect of gravitation is typically considered negligibly small at the scale of molecules but we use the term ``gravitation'' as a representative of any constant force due to a potential or mechanical pressure.
 The gravitation component significantly changes the behavior of the model as the velocity of the inert particle is no longer an increasing process and the inert particle can never escape the Brownian particle (they keep colliding). The joint behavior of the two particles is thus, a priori, far from clear. Among other things, we will show that in this battle between the gravitational pull and the Brownian push, gravitation ``wins'' as both  particles eventually ``fall'' with asymptotic velocity $-g$.

A number of related models were studied in \cite{BBCH,White,BW}. In \cite{BBCH}, reflecting diffusions were considered in bounded smooth domains in $\mathbb{R}^d$, that acquired drift proportional to the local time spent on the boundary of the domain. Product form stationary distributions were derived for the joint law of the position of the reflecting process and its drift. In \cite{White}, general classes of processes with inert drift were constructed and recurrence, transience and stationary distributions were investigated for some particular examples. In \cite{BW}, some Markov processes with discrete state spaces were studied as approximations to processes with inert drift. Necessary and sufficient conditions in order to have a stationary distribution in product form were given for these discrete state space Markov processes and it was conjectured that these conditions carry over to some models with continuous state space via appropriate limiting operations.

We will now discuss a few aspects of the model that we find intriguing. The constant $g$ enters the model as the acceleration but ends up as the asymptotic velocity for both inert and Brownian particles, $S$ and $X$
(see Theorem \ref{thm:stronglaw}). 
With the hindsight, one could provide the following ``explanation'' for this strange transformation of the role of $g$. Since the local time represents the change of position for $X$ and the change of velocity for $S$, 
it is perhaps not so surprising that the acceleration of $S$ becomes the velocity for $X$. 
Because of the parabolic drift, excursions of S above X are not very large, which makes the two particles remain close  on  large time scales, so their asymptotic velocity must be the same. We study the ``zero-noise case" (i.e. with $B_t \equiv 0$) in Remark \ref{remark:detsyst} and show that an analogous result holds for this deterministic system, which provides further evidence as to why this result might be true even in the presence of noise.

The product form of the stationary distribution for $(V, S-X)$ (see Theorem \ref{thm:stat}) came as a surprise to us but, with the hindsight, we see that the model 
\eqref{eqnarray:eqmotion} fits into the framework of \cite[Section 3]{BW}. In other words, an appropriate discretized version of 
\eqref{eqnarray:eqmotion} should satisfy \cite[Cor.~2.3]{BW}, and it might be possible to perform a limiting operation on that discretized model as conjectured in \cite{BW} and deduce the product form stationary distribution for the original model, although we do not prove this in this paper.

The variance of the first component of the stationary distribution, representing $V$, does not depend on $g$. Once again, this is not surprising with the hindsight, since the stationary distribution for the local time in a related model in \cite[Thm.~6.2]{BBCH} does not depend on the state space (Euclidean domain) for the Brownian particle. 

The rest of the article is organized as follows. 
Some of our main results are stated in Section \ref{mainres}.
Existence and uniqueness  of the strong solution to 
\eqref{eqnarray:eqmotion} is proved in Section \ref{sec:exist}. Section \ref{sec:estim} is devoted to some technical estimates. Fluctuations of the velocity and gap processes during excursions of $V$ above and below the level $-g$ are studied in Section \ref{sec:fluct}. These estimates are turned into universal fluctuation results for $V$ and $S-X$ and laws of large numbers in Section \ref{sec:strong}. Finally, the stationary distribution for the velocity and gap processes is derived in Section \ref{sec:station}.

\section{Main results}\label{mainres}

This section contains statements of those of our main results that are non-technical.

The first theorem states that $Z:=(V,S-X)$ has a unique stationary distribution which is the product of a Gaussian distribution and an exponential distribution. We will prove in Section \ref{sec:station} that the laws of $Z_t$ converge in  the total variation distance to the stationary distribution.

\begin{thm}\label{thm:stat}
The process $Z:=(V,S-X)$ has a 
unique stationary distribution
with the
density with respect to Lebesgue measure given by
\begin{align}
\label{eq:density}
\xi(v,h) = \frac{2g}{\sqrt{\pi} } e^{-2gh}e^{-(v+g)^2},
\qquad v\in \Reals, h\geq 0.
\end{align}
Furthermore, $Z_t$ converges to this distribution in total variation distance as $t \rightarrow \infty$. 
\end{thm}

The next result shows that the fluctuations of the velocity process are of the order $\sqrt{\log t}$ while the fluctuations of the ``gap process'' $S_t - X_t$ are of the order $\log t$.

\begin{thm}\label{thm:velfluc}
For any $Z_0=z$, almost surely,
\begin{align}\label{equation:flucplus}
 -\liminf_{t \rightarrow \infty}\frac{V_t}{\sqrt{\log t}}
 &=\limsup_{t \rightarrow \infty}\frac{V_t}{\sqrt{\log t}} =1,\\
\label{equation:gapplus}
 \limsup_{t \rightarrow \infty}\frac{S_t-X_t}{\log t} &= \frac1{2g},\\
\label{equation:gapminus}
\liminf_{t \rightarrow \infty}\frac{S_t-X_t}{\log t} &= 0.
\end{align}
\end{thm}

We will show that both the ``inert particle'' $S_t$ and the reflected Brownian particle $X_t$ behave like Brownian motion with constant negative drift $-g$ on the large scale and, therefore, they satisfy the same Strong Law of Large Numbers
as Brownian motion with drift. The next theorem gives precise estimates on the oscillations of $S$ and $X$ from Brownian motion with drift $-g$. The strong law follows as a consequence. 

\begin{thm}\label{thm:stronglaw}
For any $Z_0=z$, almost surely,
\begin{align*}
 - \liminf_{t \rightarrow \infty} \frac{X_t-(B_t-gt)}{\sqrt{\log t}} 
 &=\limsup_{t \rightarrow \infty} \frac{X_t-(B_t-gt)}{\sqrt{\log t}}  =1,\\
 \limsup_{t \rightarrow \infty} \frac{S_t-(B_t-gt)}{\log t} & =\frac1{2g},\\
\liminf_{t \rightarrow \infty} \frac{S_t-(B_t-gt)}{\log t}&=0,
\end{align*}
It follows that, a.s.,
\begin{equation*}
\lim_{t \rightarrow \infty}\frac{X_t}{t}=\lim_{t \rightarrow \infty}\frac{S_t}{t}=-g.
\end{equation*}
\end{thm}

\begin{remark}\label{remark:detsyst}
To gain some insight into why one might expect $\lim_{t \rightarrow \infty}\frac{X_t}{t}=\lim_{t \rightarrow \infty}\frac{S_t}{t}=-g$ to hold almost surely, we consider the ``zero-noise case", i.e., the deterministic two-particle system driven by \eqref{eqnarray:eqmotion} with $B_t \equiv 0$. Suppose we start from the initial conditions $S_0=X_0, V_0 <0$. Let $\tau = \inf\{t>0: V_t=0\}$. Then on $[0, \tau]$, $S_t$ is decreasing and one can conclude from the Skorohod equation (see  \cite[Lem.~6.14, Ch.~3]{karatzas}) that $L_t= \sup_{u \le t}(S_0-S_u) = S_0-S_t$. Using this in \eqref{eqnarray:eqmotion}, we obtain
\begin{align}\label{align:detsol}
S_t = X_t, \ \ \ V_t = -g + (V_0+g)e^{-t} \ \ \ \text{ for all } t \le \tau.
\end{align}
The above implies $\tau= \infty$ and $V_t \rightarrow -g$ as $t \rightarrow \infty$. Thus,
\begin{equation}\label{equation:detlimit}
\lim_{t \rightarrow \infty}\frac{X_t}{t}=\lim_{t \rightarrow \infty}\frac{S_t}{t}=\lim_{t \rightarrow \infty}\frac{1}{t}\int_0^tV_udu=-g
\end{equation}
holds for the zero-noise case when $S_0=X_0, V_0 <0$. If $S_0>X_0$ and $V_0 \in \mathbb{R}$, it follows from \eqref{eqnarray:eqmotion} that if $\sigma = \inf\{t \ge 0: S_t=X_t\}$, then $\sigma < \infty$ and $V_{\sigma}<0$, and thus, we can perform the same computations for $t>\sigma$ to deduce that \eqref{equation:detlimit} holds.

Finally, suppose $S_0=X_0$ and $V_0 \ge 0$. If $S_t=X_t$ for all $t \ge 0$, we use \eqref{eqnarray:eqmotion} to obtain $dL_t=-dX_t=-dS_t = -V_tdt$ which gives us $dV_t=-V_tdt - gdt$. This yields \eqref{align:detsol} and consequently \eqref{equation:detlimit}. Otherwise, there exists $t_0>0$ such that $S_{t_0}>X_{t_0}$ and the previous calculations again yield \eqref{equation:detlimit}. Thus, we see that $\lim_{t \rightarrow \infty}\frac{X_t}{t}=\lim_{t \rightarrow \infty}\frac{S_t}{t}=-g$ always holds in the zero-noise case. As this is a ``law of large numbers" type result, it is natural to expect that this result would also hold with the Brownian noise via some scaling properties of Brownian motion. In the proof of Theorem \ref{thm:stronglaw}, however, we derive the fluctuation results quite differently via some technical estimates and derive \eqref{equation:detlimit} as a consequence of these results.
\end{remark}

\medskip
We will use the following notation: $a\land b = \min(a,b)$ and $a\lor b = \max(a,b)$.

\section{Existence {and uniqueness} of the process}\label{sec:exist}

{In this section, we prove the existence and pathwise uniqueness of solutions to \eqref{eqnarray:eqmotion}. We also prove the well-posedness of the submartingale problem corresponding to the process. The latter fact will be an essential ingredient in proving the existence of the stationary distribution in Section \ref{sec:station}.}

\begin{remark}\label{oc18.1}
The following translation invariance property follows immediately from the 
form of equations \eqref{eqnarray:eqmotion}. Consider real numbers $x,s,v,l$ with $s \ge x$ and $l \ge 0$. If  $\{(X_t,S_t,V_t,L_t), t\geq 0\}$ solves 
\eqref{eqnarray:eqmotion} with the initial conditions
$(X_0,S_0,V_0,L_0)=(0,s-x,v,0)$ then $\{(\wh X_t,\wh S_t,\wh V_t, \wh L_t), t\geq 0\} := \{(x+X_t,x+S_t,V_t,\ell+L_t), t\geq 0\}$ is a solution to \eqref{eqnarray:eqmotion} with the initial conditions $(\wh X_0,\wh S_0,\wh V_0,\wh L_0)=(x,s,v,\ell)$. Because of this, we will always assume in our technical estimates that $B_0=X_0=L_0=0$, unless explicitly stated otherwise.
\end{remark}

The process given by $H_t=S_t-X_t$ will be called the gap process. If we know both $V$ and $H$, we can recover the movement of the individual particles by first integrating $V$ to obtain $S$, and then computing $X_t=S_t-H_t$. 

 Thus, existence and uniqueness  of a strong solution to the system \eqref{eqnarray:eqmotion} are equivalent to those of  the following system of equations:
\begin{align}\label{eq:z}
\begin{cases}
d V_t = dL_t - g dt,  \\
dH_t = - dB_t + V_t dt + d L_t,
\end{cases}
\end{align}
where $B_t$ is a standard one dimensional Brownian motion, $H_t \geq 0$ for all $t\geq 0$,  and $L_t$ is a continuous, non-decreasing process satisfying $dL_t = \Indicator_{\{0\}}(H_t)dL_t$. As before, we will write $Z_t=(V_t,H_t)$.

If $B_t$ and $Z_t$ are given, then $L_t$ can be computed from the equation $L_t = L_0+ V_t - V_0 + g t$. Thus, the complete description of the strong solution to \eqref{eq:z} can be given  in terms of only $Z$ and $B$.

Even though reflected diffusions in $\HH=\Reals\times\Reals_+$ are well studied and many classical results are available, we have not found a direct reference for existence and uniqueness of equations \eqref{eq:z}, mainly due to two technical issues: (i) $Z_t$ is not a strictly elliptic diffusion,
and (ii) the drift vector $(-g,V_t)^T$ is unbounded in the $v$-component. We will split our proof of existence and uniqueness into two lemmas: one that shows that a local solution exists, and another that extends local solutions to global ones. The second lemma will also be used to show that the submartingale problem is well posed.

In the following, we will write
$$
\L = \frac12\partial_{hh} +v\partial_h -g\partial_v
$$
for the second order differential operator associated to the generator of the Markov process $Z_t$ satisfying \eqref{eq:z} (provided it exists) in the interior of the upper half plane $\HH$.

\begin{lem}
\label{le:z_existence}
For each $N\geq 0$, there is a weak solution $Z^{N}_{t} = (V_{t}^N,H^{N}_{t})$   of \eqref{eq:z} up to time $T^{N}=\inf\kpar{t>0 : |Z^{N}_{t}| > N}$. Also, \eqref{eq:z} satisfies pathwise uniqueness up to time $T^N$.
\end{lem}

\begin{proof}
Note that equation \eqref{eq:z} can be written as
\begin{align}
\label{eq:zz}
dZ_{t}^T &= 
\begin{pmatrix}
0 & 0 \\ 0 & -1
\end{pmatrix}
dB^{*}_{t} +
\begin{pmatrix}
-g \\ V_{t}
\end{pmatrix}
dt +
\begin{pmatrix}
1 \\ 1
\end{pmatrix}
dL_{t},
\end{align}
where $B^{*}_{t} = (W_{t},B_{t})^{T}$ and $W_{t}$ is a standard, one dimensional Brownian motion independent of $B$, that has no effect on the paths of $Z_{t}$, that is,  if we replace $W_t$ with another Brownian motion $W'_t$ then the same process $Z_t$ will solve \eqref{eq:zz}. We will now apply a standard localization technique. Note that if a weak solution $Z^{N}_t$ exists for a modified version of \eqref{eq:zz} with the drift vector replaced by $(-g,V_{t}\wedge N)^{T}$, then $Z^{N}_t$ will also be a weak solution to the original equation \eqref{eq:zz} up to time $T^{N}$. The drift and diffusion coefficients for the modified version of \eqref{eq:zz} are Lipschitz and bounded and thus, the equation fits into the setup of Theorem 1 in  \cite{Wat71}: the diffusion matrix $a=\sigma\sigma^T$ has component $a_{22}=1$, the drift vector is Lipschitz and bounded, and the reflection vector $\gamma:=(1,1)^T$ is constant with the unit length component in the direction of the normal to the boundary. Even though the statement of Theorem 1 in  \cite{Wat71} is about weak existence and weak uniqueness, its proof actually shows pathwise uniqueness. This implies the two claims made in the lemma.
\end{proof}

\begin{lem}
\label{le:zz}
For each $N\geq 1$, let $Z^N_t$ be a weak solution up to time $T^{N}=\inf\kpar{t>0 : |Z^{N}_{t}| > N}$ of \eqref{eq:z}, with $Z^N_0 = z$. Then
\begin{enumerate}[(a)]
\item There are constants $C_{1},C_{2}>0$, independent of $N$, such that $\EE\rpar{|Z^{N}_{t \wedge T^N}|^{2}}\leq C_{1}e^{C_{2}t}$.

\item All processes $Z^N_t$ can be chosen to be strong solutions of \eqref{eq:z}. They can be constructed so that $Z^N_t = Z^{N'}_t$ for $t\leq T^{N}$ if $N\leq N'$. Hence, they can be extended to a strong solution $Z_t$ up to time $T^* := \sup_{N} T^{N}$. Pathwise uniqueness holds for $Z_t$ up to time $T^*$.

\item $T^* =\infty$ a.s.
\end{enumerate}
\end{lem} 

\begin{proof}

To prove (a), set $\eta(v,h)=2h^2+v^2-2hv$, and recall that $\L\eta = \frac12\partial_{hh}\eta +v\partial_h \eta -g\partial_v \eta$
and $\gamma = (1,1)^T$. It is elementary to check that $\frac13(h^2+v^2)\leq \eta(v,h)\leq3(h^2+v^2)$. Since $|Z^{N}_t|<N$ for $t < T^{N}$, we have by \Ito's formula
\begin{align*}
\frac13\EE\rpar{\left|Z^{N}_{t\wedge T^N}\right|^2}\leq \EE\eta(Z^{N}_{t\wedge T^N}) &= \eta(z) + \EE\int_0^{t\wedge T^N} \L\eta(Z_u^{N}) du 
+ \EE \int_0^{t\wedge T^N} \nabla\eta(Z^{N}_u)^T\gamma dL^{N}_u \\
&= \eta(z) +
\EE\int_0^{t\wedge T^N} \L\eta(Z_u^{N}) du ,
\end{align*}
since $ \nabla\eta(Z_u^{N})^T\gamma dL^{N}_u =2H^{N}_udL^{N}_u=0$. To bound the integral on the right hand side, we note that there exist constants $K_1,K_2>0$ such that $|\L\eta(z)|\leq K_1 + K_2|z|^2$. 
Putting all these inequalities together we obtain
\begin{align*}
\EE\rpar{\left|Z^{N}_{t\wedge T^N}\right|^2} &\leq 9 { |z|^2} + 3K_1t + 3K_2\int_0^t \EE\rpar{|Z_{u\wedge T^N}^N|^2} du.
\end{align*}
Now (a) follows from Gronwall's inequality.

Since pathwise uniqueness holds for \eqref{eq:z} up to time $T^{N}$, we can apply a well-known argument by Yamada and Watanabe (see  \cite[Ch.~IV, Thm.~11]{IkW89} or  \cite[Ch.~5, Corollary~3.23]{karatzas}) with minor modifications to show that $Z^N_{t}$ can be chosen as a strong solution up to time $T^N$.  By pathwise uniqueness, if $M>N$ we have that $Z^M_t = Z^N_t$ for $t < T^{N}$. 
This allows us to define $Z_t=Z^N_t$ for $t<T^{N}$ in a consistent way, and thus obtain a strong solution for $t<T^{*}=\sup_N T^{N}$. It is clear that pathwise uniqueness also holds  up to time $T^{*}$. This shows (b). 

It remains to show that $T^{*}=\infty$ almost surely. For any $\alpha>0$ and  $R>0$, by  (a), we have
$$
\PP\rpar{ T^* \leq\alpha} \leq \EE\rpar{\frac{\abs{Z^R_{\alpha}}^2}{R^2} \Indicator_{\kpar{T^{R} \leq\alpha}}} \leq \frac{C_1}{R^2} e^{C_2\alpha}.
$$
Taking $R\to\infty$ on the right hand side, we conclude that $\PP\rpar{ T^* \leq\alpha} =0$ for each $\alpha>0$, and thus $T^* = \infty$ a.s.
\end{proof}

As a direct consequence of the two previous lemmas, we are able to show existence and pathwise uniqueness for equation \eqref{eq:z}, which we record in the following theorem.

\begin{thm}
\label{thm:existence_sde}
The system of stochastic equations \eqref{eq:z} has a square integrable, strong solution $(V_t,H_t)$, and satisfies pathwise uniqueness.
\end{thm}

Our next theorem will be used in the derivation of the stationary distribution of $Z_t$. Recall that $\HH=\Reals\times\Reals_+$. We will denote by  $\C$  the set of continuous functions from $[0,\infty)$ to $\overline\HH$. We denote by $C^k_0(A)$  the set of compactly supported functions on $A$ with continuous derivatives up to order $k$, allowing $k=\infty$.  We will use $\B$ to denote the Borel $\sigma-$algebra of $\C$, and $\kpar{\F_t}$ will denote the natural filtration on $\C$.
Recall that $\gamma=(1,1)^T$ and  $\L f(v,h) = \frac12 \partial_{hh}f +v\partial_h f -g\partial_v f $.

\begin{thm}
\label{th:submg}
The submartingale problem for $(\L,\gamma)$ in $\HH$ is well-posed, that is, there is a unique family of measures $\kpar{\PP_z : z\in \overline{\HH}}$ on $\rpar{ \C,\B }$ such that, for each $z\in \overline{\HH}$, the following properties hold
\begin{enumerate}
\item $\PP_z(\omega(0) = z)=1$, and $\PP_z(\omega(t)\in\overline\HH)=1$.
\item For $t\geq 0$, and each $f\in C_0^2(\Reals^2)$ such that $\nabla f(y)^T\gamma\geq 0$ for all $y\in \partial \HH$, the process
\begin{align}
\label{eq:f_submg}
S_t[f] :=  f(\omega(t)) - \int_0^t \L f(\omega(u))du
\end{align}
is a submartingale in $\rpar{\PP_z,\C,\B,\kpar{\F_t}}$.
\end{enumerate}

Moreover, the unique solution to the submartingale problem corresponds to the law of the process $Z_t=(V_t,H_t)$ solving \eqref{eq:z}.
\end{thm}

Before proceeding to the proof of Theorem \ref{th:submg}, we will prove an important property about the amount of time the process associated to a solution to the submartingale problem spends on the boundary of the domain.

\begin{lem}
\label{le:indicator}
Let $Z_t^* = (V^*_t,H^*_t)$ be a solution of the submartingale problem for $(\L,\gamma)$ in $\HH$. Then
\begin{align*}
\int_0^t \Indicator_{\partial\HH}(Z_u^*)du =0 \qquad\text{a.s.}
\end{align*}
\end{lem}

\begin{proof} We will write $z=(v,h)$ to simplify the notation. 
For $n> 0$, define $q_n(h) = h^2\exp(-nh)$. Note that $q_n(h)=\partial_hq_n(h)=0$ on $\partial\HH$. For $N>0$, let $\varphi_N$ be a $C^2_0(\Reals^2)$ function satisfying: $0\leq\varphi_N(z)\leq 1$, $\varphi_N(z)=1$ for $\abs{z}\leq N$, and $\varphi_N$ has uniformly (in $N$) bounded derivatives up to the second order. Let $\tilde q_{n,N}(z) = q_n(h)\varphi_N(z)$. Since $ q_{n}(h)$ is bounded above by $4e^{-2}/n^2$, it is clear that 
\begin{align}\label{align:mart1}
\lim_{n\to\infty} \tilde q_{n,N}(z) = 0, \ \ \lim_{n\to\infty} \EE\rpar{\tilde q_{n,N}(Z^*_t)} = 0 \ \ \text{for any } t\geq 0.
\end{align}
We have,
\begin{align*}
\int_0^t 
\L \tilde q_{n,N}(Z^*_u) du 
&= \int_0^t \left(\varphi_N(Z^*_u)\L  q_{n}(H^*_u) +  q_{n}(H^*_t) \L \varphi_N(Z^*_u) + \partial_h  q_{n}(H^*_t)\partial_h\varphi_N(Z^*_u)\right)du.
\end{align*}
We will argue that the integrals of the second and third terms on the right hand side go to zero as $n\to \infty$, for every fixed $N$.
The claim holds for the second term because  $ q_{n}(h)$ is bounded above by $4e^{-2}/n^2$ and $\L\varphi_N(z)(h) $ is uniformly bounded.
The function $\partial_h q_n(h)$ converges to zero, and is uniformly bounded in $n$. The function $\partial_h\varphi_N(Z^*_u)$ is uniformly bounded. These observations prove the claim for the third term. 

We now turn our attention to the first term on the right hand side.
The function $\varphi_N(z)\L q_n(h) $ is uniformly bounded and converges to $\varphi_N(z)\Indicator_{\kpar{0}}(h)$ as $n$ goes to infinity. Therefore, applying the dominated convergence theorem, we see that
\begin{align}\label{align:mart2}
\lim_{n\to\infty} \EE\rpar{ \int_0^t \L \tilde q_{n,N}(Z^*_u) du } = \EE\rpar{ \int_0^t \varphi_N(Z^*_u) \Indicator_{\partial\HH}(Z^*_u) du}.
\end{align}
Since $\left(\nabla \tilde q_{n,N}(z)\right)^T\gamma = 0$ on $\partial\HH$ we have that $S_t[\tilde q_{n,N}] - S_0[\tilde q_{n,N}]$ is a martingale (this can be proved by considering the submartingale problem applied to $\tilde q_{n,N}$ and $-\tilde q_{n,N}$). Therefore,
$$
\EE\rpar{\tilde q_{n,N}(Z^*_t)} = \tilde q_{n,N}(z) + \EE\rpar{ \int_0^t \L \tilde q_{n,N}(Z^*_u) du }.
$$
Letting $n \rightarrow \infty$ in the above equation and using \eqref{align:mart1} and \eqref{align:mart2}, we obtain
$$
\EE\rpar{ \int_0^t \varphi_N(Z^*_u) \Indicator_{\partial\HH}(Z^*_u) du} =0.
$$
The lemma now follows from the monotone convergence theorem upon taking $N \rightarrow \infty$.
\end{proof}

\begin{proof}[Proof of Theorem \ref{th:submg}]
Using It\^o's formula, it is straightforward to check that solutions $Z$ to \eqref{eq:z}, for different initial values $Z_0=z$, constitute a family that solves the submartingale problem for $(\L,\gamma)$ in $\HH$. It only remains to prove the uniqueness in law of this solution. To this end, we will show that any solution $Z^*$ to the submartingale problem is a weak solution of \eqref{eq:z}.

Consider a solution $Z^*_t=(V^*_t,H^*_t)=\omega(t)$  to the submartingale problem \eqref{eq:f_submg} with $Z^*_0=z\in\overline\HH$. 
 We will use Theorem 2.4 of \cite{StV71}. That paper is concerned with processes whose diffusion coefficients $a_{ij}$ are strictly elliptic and drift coefficients $b_i$ are bounded (see page 147, and (i') and (ii') on page 159). Our diffusion coefficients are not elliptic and our drift coefficients are not bounded, so,  the results from the cited paper do not apply directly to our setting.  However, these assumptions are used neither in the definition of the class $F$ in \cite[page 161]{StV71}, nor in the statements and proofs of Lemmas 2.3, 2.4, and 2.5, and Theorem 2.4 in that paper. Also, the definition of the submartingale problem in \cite{StV71} is slightly different than ours, because, in their setting, the integral in \eqref{eq:f_submg} has the indicator function of $\HH$ as a factor in the integrand, which is not an issue in view of Lemma \ref{le:indicator} above. In order to use Theorem 2.4 of \cite{StV71}, we will show  that all functions $ f\in C_0^2(\Reals^2)$ belong to the class $F$, by modifying an argument from \cite{StV71}. 
 
We will briefly sketch the underlying idea of the proof. Theorem 2.4 of \cite{StV71} shows (without using ellipticity or boundedness of drift) that there is a unique, continuous, non-decreasing, non-anticipating ``local time" $L^*$, such that
$$
M_t[f] = f(Z^*_t) - \int_0^t \L f(Z^*_u)du - \int_0^t \left(\nabla f(Z^*_u)\right)^T\gamma\ dL^*_u
$$ 
is a martingale for each $f \in F$ (with $F$ taken as the class of functions defined in \cite[page 161]{StV71}). We will first show that every $f \in C_0^2(\Reals^2)$  belongs to the class $F$. Then, by appropriate choices of $f \in C_0^2(\Reals^2)$, it will be shown that any solution $Z^*$ to the submartingale problem is also a weak solution of \eqref{eq:z} with $L$ taken as $L^*$ determined by the class of functions $F$ as described above.
 
We will first show that (i) and (ii) in \cite[page 161]{StV71} are satisfied by any $f\in C_0^2( \Reals^2)$. To see this, fix an arbitrary $f\in C_0^2( \Reals^2)$.  First, suppose that $f$  has support in $\HH$, so that $(\nabla f)^T\gamma= 0$ on $\partial\HH$. Hence, $S_t[f]$ is a martingale. For general $f\in C_0^2( \Reals^2)$, let $\HH_n = \kpar{(v,h) : h>n^{-1}, |v|<n}$ and consider $\eta_n\in C_0^\infty( \Reals^2)$ such that $\eta_n =1$ in $\HH_n$ and $\eta_n=0$ outside of $\HH_{n+1}$. Define the ``smooth localization" $f_n \in C_0^2(\Reals^2)$ of $f$ as $f_n=f\eta_n$. For any stopping time $\tau$, set $\tau_n=\tau\wedge n$, and $\tau_n'= \inf\kpar{t\ge \tau_n: Z_t^*\notin\HH_n}\wedge n$. It follows that $S_{t\wedge\tau'_n}[f_n] - S_{t\wedge\tau_n}[f_n]$ is a martingale. But 
$$
S_{t\wedge\tau'_n}[f] - S_{t\wedge\tau_n}[f] = S_{t\wedge\tau'_n}[f_n] - S_{t\wedge\tau_n}[f_n].
$$
By taking $n\uparrow\infty$, we obtain that $S_t[f]$ is a local martingale in $\HH$, in the sense of \cite[page 158]{StV71}. Since $Kf =\L f$ is bounded for any $f\in C_0^2( \Reals^2)$, we see that (i) and (ii) in \cite[page 161]{StV71} are satisfied.

Now, we will show that (iii) and (iv) in \cite[page 161]{StV71} are satisfied by any $f\in C_0^2( \Reals^2)$. If $f\in C_0^2( \Reals^2)$ satisfies $(\nabla f)^T\gamma\geq 0$ on $\partial\HH$, then from the statement of the submartingale problem above, $S_t[f]$ is a locally bounded, continuous submartingale. By the Doob-Meyer decomposition theorem, it follows that there is an integrable, non-decreasing, non-anticipating continuous function $L_t[f]$ such that $L_0[f]=0$ and $S_t[f]-L_t[f]$ is a martingale. For general $f \in C_0^2( \Reals^2)$, consider $\phi(v,h)= \arctan(h)$, which is a defining function for $\HH$ in the sense of \cite{StV71} (see page 158),  set $\alpha=-\inf\kpar{(\nabla f(y))^T\gamma : y\in\partial\HH}$, and define $f_\alpha=f+\alpha\phi$. It is clear that $\left(\nabla f_\alpha\right)^T\gamma\geq 0$, and we can define $L_t[f] = L_t[f_\alpha] - \alpha L_t[\phi]$. We obtain that $L_t[f]$ is a non-anticipating continuous function of bounded variation such that $L_0[f]=0$, $\EE(|L_t[f]|) \leq   \EE(L_t[f_\alpha]) + \alpha \EE(L_t[\phi]) <\infty$, and $S_t[f] - L_t[f]$ is a martingale. This shows (iii) and (iv) in \cite[page 161]{StV71}. This proves that $f$ belongs to the class $F$ for any $f\in C_0^2( \Reals^2)$.

 Hence, we can apply \cite[Thm.~2.4]{StV71} to see that there exists a  unique, continuous, non-decreasing, non-anticipating process $t\mapsto L^*_t$, such that $L^*_0=0$, $\EE(L^*_t) <\infty$, $dL^*_t = \Indicator_{\partial\HH}(Z^*_t)dL^*_t$, and 
\begin{align}
\label{eq:f_mg}
M_t[f] = f(Z^*_t) - \int_0^t \L f(Z^*_u)du - \int_0^t \left(\nabla f(Z^*_u)\right)^T\gamma\ dL^*_u
\end{align}
is a martingale for each $f\in F$, and in particular, for $f\in C_0^2( \Reals^2)$.

Using that $\L f^2(z) = 2f(z)\L f(z) +  |\partial_h f(z)|^2$, we obtain
\begin{align*}
M_t[f^2] &= f(Z^*_t)^{2} -  \int_0^t |\partial_h f(Z^*_u)|^2 du - 2\int_0^t   f(Z^*_u)\L f(Z^*_u)du - 2\int_0^t f(Z^*_u)\left(\nabla f(Z^*_u)\right)^T\gamma\ dL^*_u.
\end{align*}%
Using \eqref{eq:f_mg} to compute $df(Z^*_t)$, we see that
\begin{align*}
\int_0^t f(Z^*_u) df(Z^*_u) &= \int_0^t f(Z^*_u) dM_u[f] + \int_0^t f(Z^*_u) \L f(Z^*_u) du + \int_0^t f(Z^*_u)\left(\nabla f(Z^*_u)\right)^T\gamma  dL^*_u.
\end{align*}
It follows that
\begin{align*}
f(Z^*_t)^2 - 2 \int_0^t f(Z^*_u) d f(Z^*_u) 
&= M_t[f^2] - 2\int_0^t f(Z^*_u) dM_u[f] +  \int_0^t |\partial_h f(Z^*_u)|^2 du .
\end{align*}
It\^o's formula shows that the left hand side in the equation above equals to $f(Z^*_0)^2 + \ppar{f(Z^*)}_t$, where the bracket $\ppar{\cdot}$ stands for quadratic variation. The right hand side has two continuous martingales plus a continuous  process of bounded variation. By uniqueness of the decomposition of continuous semimartingales, we conclude that
\begin{align}
\label{eq:qv}
\ppar{f(Z^*)}_t =   \int_0^t |\partial_h f(Z^*_u)|^2 du.
\end{align}
Since $f(Z^*_t)-M_t[f]$ is continuous with bounded variation, we see that $\ppar{M[f]}_t  = \ppar{f(Z^*)}_t$ is also given by \eqref{eq:qv}. This formula  suggests that there is a Brownian motion $B^*$ such that $dM_t[f] = \partial_h f(Z^*_t)dB^*_t$ for all $f\in C_0^2( \Reals^2)$. We proceed to prove this by localization.

\medskip

Let $T^N = \inf\kpar{t\geq 0 : |Z^*_t|>N}$, and $Z^N_t = Z^*_{t\wedge T^N}$. For $\mu,\lambda\in\Reals$,  let $f$ be a $C_0^2( \Reals^2)$ function such that $f(v,h)=\mu h+\lambda v$ for $\abs{(v,h)}<N$. From \eqref{eq:qv} we obtain $\ppar{\mu H^* + \lambda V^*}_{t\wedge T^N} =  \mu^2\cdot\rpar{t\wedge T^N}$. By Levy's characterization theorem, there is a Brownian motion $B^N_t$ (in a possibly enlarged probability space) such that $M_t[f] =  M_0[f] - \mu B^N_t$ for $t<T^N$. Unravelling our definitions and using \eqref{eq:f_mg} for $f$ at time $t\wedge T^N$ we obtain
\begin{align*}
\mu H^*_{t\wedge T^N} + \lambda V^*_{t\wedge T^N} = \mu H^*_0 + \lambda V^*_0 - \mu B^N_{t\wedge T^N} + \int_0^{t \wedge T^N} (\mu V^*_u -g\lambda) du + (\mu+\lambda) dL^*_{t\wedge T^N}.
\end{align*}
From this, it is direct to see that  for each $N\geq 0$, $Z^N_t$ is a weak solution to \eqref{eq:z} up to time $T^N$. It follows from Lemma \ref{le:zz}  that $Z^*$ is a  weak solution to \eqref{eq:z}. Since this equation satisfies pathwise uniqueness by Theorem \ref{thm:existence_sde}, it also satisfies uniqueness in law (\cite[Ch.~5, Proposition~3.20]{karatzas} with minor modifications for the reflected case), which shows that there is a unique solution to the submartingale problem.
\end{proof}

\section{Hitting time estimates}\label{sec:estim}

{In this section, we derive some preliminary estimates for hitting times of $V$ and $S-X$. These will be essential in most of the calculations leading to fluctuation results, strong laws and convergence to stationarity.}
 
We will use $\lfloor\, \cdot\, \rfloor$ to denote the greatest integer function.
We will write $C, C', C'', \dots$ for finite positive constants, whose values might change from line to line.

Recall that $H_t = S_t-X_t$.
Let
\begin{align*}
\vtau_{a}&=\inf\{t\ge 0: V_t=a\},\\
\stau_{a}&=\inf\{t\ge 0: H_t=a\},\\
\tau^{B,c}_a &= 
 \inf\{t\geq 0: B_t + c t = a\}, \\ 
\sigma(u)&=\inf\{t \ge u: S_t=X_t\},
\end{align*}
where $a,c \in \Reals$, $u \ge 0$,  with the convention that $\inf \emptyset = \infty$. 

\begin{remark}\label{oc16.1}
If $a>V_0$, then $S_{\vtau_a}=X_{\vtau_a}$, as otherwise, by path continuity of $S$ and $X$, there will exist a small time interval $[\vtau_a-\delta, \vtau_a]$ for some $\delta>0$ such that $S_u> X_u$ for all $u \in [\vtau_a-\delta, \vtau_a]$ and consequently, the velocity will be strictly decreasing in this interval, which is a contradiction to $\vtau_a$ being the first hitting time of level $a$ by the velocity process $V$. It is not necessarily true that $S_{\vtau_a}=X_{\vtau_a}$ for $a< V_0$.
\end{remark}

\begin{remark}\label{oc31.1}
For any initial values $V_0=v$, $X_0=x$ and $S_0=y$, $\sigma(0) < \infty$, a.s. To see this, note that on the event $\{\sigma(0)=\infty\}$, the trajectory of $S$ is a downward parabola  and the trajectory of $X$ is the trajectory of Brownian motion $B$ shifted by a constant, and staying forever under the parabola. This event has zero probability because $B_t/t\to 0$, a.s. 
\end{remark}

It is elementary to check that if $S_0\ge B_0 =X_0$, then the local time satisfies the usual Skorohod equation (see  \cite[Lem.~6.14, Ch.~3]{karatzas}),
\begin{align}\label{oc9.1}
L_t= 0\lor \sup_{u \le t}(B_u-S_u ), \qquad t\geq 0.
\end{align}


We will be frequently approximating the local time $L$ by using the local time of standard Brownian motion $B$ reflected, via the Skorohod equation, downward on a line of slope $a$ passing through the origin. We will use
the following notation,
\begin{align}\label{oc15.1}
L^{(a)}_t=\sup_{u \le t}(B_u - au).
\end{align}

We will use the following well known formulas (see 
\cite[Ch.~2, (9.20); Ch.~3, (5.12) and (5.13)]{karatzas}).
If $B_0=0$ then
\begin{align}
\label{oc13.3}
&\Prob\left(\sup_{s\leq t} B_s \geq x\right)
= 2 \int_x^\infty \frac 1 {\sqrt{2 \pi t}} e^{-u^2 / (2t)} du
\leq \frac {2\sqrt{t}} {\sqrt{2\pi}x} e^{- x^2/(2t)}, \quad t,x>0,\\
\label{oc13.1}
&\Prob\left(\tau^{B,m}_a \in dt\right) =
\frac {|a|}{\sqrt{ 2 \pi t^3}} \exp ( - (a- m t)^2/(2t))\, dt,
\quad t\geq 0,\\
&\Prob\left(\tau^{B,m}_a  < \infty\right) = \exp (ma - |ma|). \label{oc13.2}
\end{align}

The following two lemmas contain estimates for the hitting times of different levels by the velocity process $V_t$, for starting points in different ranges of values.
\begin{lem}\label{lem:hittimeest}
Assume that $H_0=0$. Then
for $0<a_1<a_2$, and $t\ge 2(a_2-a_1)/a_1$,
\begin{equation}\label{equation:hittime}
\Prob(\vtau_{-g-a_1}>t \mid V_0=-g-a_2) \le \frac{4(a_2-a_1)}{((a_2-a_1)+a_1t)a_1\sqrt{2\pi t}}e^{-a_1^2t/8}.
\end{equation}
\end{lem}

\begin{proof}
 It follows easily from \eqref{oc9.1}-\eqref{oc15.1} that, assuming that $V_0=-g-a_2$ and
 $t<\vtau_{-g-a_1}$, we have $\displaystyle{L_t \ge L^{(-g-a_1)}_t=\sup_{u\le t}(B_u+(g+a_1)u)}$. 
We will use similar inequalities between $L_t$ and $L^{(m)}_t$ later in the paper a number of times, without explicitly referring to \eqref{oc9.1}-\eqref{oc15.1}.
We have,
\begin{eqnarray*}
\Prob(\vtau_{-g-a_1}>t \mid V_0=-g-a_2)&=&\Prob(L_s-gs \leq a_2-a_1 \text{ for } s\leq t)\\
&\le&\Prob\left(\sup_{u\le s}(B_u+(g+a_1)u)-gs \leq a_2-a_1 \text{ for } s\leq t\right)\\
&\le&\Prob((B_s+(g+a_1)s)-gs \leq a_2-a_1 \text{ for } s\leq t)\\
&=&\Prob(B_s+a_1s \leq a_2-a_1 \text{ for } s\leq t) \\
&=&\Prob\left( \tau^{B,a_1}_{a_2-a_1} > t \right).
\end{eqnarray*} 
This and  \eqref{oc13.1} imply that,
\begin{align*}
\Prob(\vtau_{-g-a_1}>t \mid V_0=-g-a_2)
& \leq
\int_t^{\infty}\frac{a_2-a_1}{\sqrt{2\pi u^3}}
\exp\left(-\frac{((a_2-a_1)-a_1u)^2}{2u}\right)du.
\end{align*}
Making a change of variable from $u$ to $z=\frac{a_2-a_1}{\sqrt{u}}-a_1\sqrt{u}$, and using the fact that $t\leq u$, we see that
\begin{align*}
&\Prob(\vtau_{-g-a_1}>t \mid V_0=-g-a_2)
 \leq
\int_t^{\infty}\frac{a_2-a_1}{\sqrt{2\pi u^3}}
\exp\left(-\frac{((a_2-a_1)-a_1u)^2}{2u}\right)du\\
 &=
(a_2-a_1) \int_{-\infty}^{\frac{a_2-a_1}{\sqrt{t}}-a_1\sqrt{t}}
u^{-3/2} \left( \frac 1 2 (a_2 - a_1) u^{-3/2} + \frac 1 2 a_1 u^{-1/2}\right) ^{-1}
\frac{1}{\sqrt{2\pi}}e^{-z^2/2}dz\\
 &=
(a_2-a_1) \int_{-\infty}^{\frac{a_2-a_1}{\sqrt{t}}-a_1\sqrt{t}}
 \left( \frac 1 2 (a_2 - a_1)  + \frac 1 2 a_1 u\right) ^{-1}
\frac{1}{\sqrt{2\pi}}e^{-z^2/2}dz\\
& \leq
\frac{2(a_2-a_1)}{(a_2-a_1)+a_1t}\int_{-\infty}^{\frac{a_2-a_1}{\sqrt{t}}-a_1\sqrt{t}}\frac{1}{\sqrt{2\pi}}e^{-z^2/2}dz.
\end{align*}
Note that, when $t\ge 2(a_2-a_1)/a_1$, $\frac{a_2-a_1}{\sqrt{t}} \le \frac{a_1\sqrt{t}}{2}$. Thus, the last estimate and \eqref{oc13.3} yield
\begin{align*}
&\Prob(\vtau_{-g-a_1}>t \mid V_0=-g-a_2)
 \leq
\frac{2(a_2-a_1)}{(a_2-a_1)+a_1t}\int_{-\infty}^{\frac{a_2-a_1}{\sqrt{t}}-a_1\sqrt{t}}\frac{1}{\sqrt{2\pi}}e^{-z^2/2}dz\\
& \leq
\frac{2(a_2-a_1)}{(a_2-a_1)+a_1t}\int_{-\infty}^{-a_1\sqrt{t}/2}\frac{1}{\sqrt{2\pi}}e^{-z^2/2}dz
\leq 
\frac{2(a_2-a_1)}{(a_2-a_1)+a_1t} \cdot
\frac{2}{\sqrt{2\pi}a_1\sqrt{t}}e^{-a_1^2t/8},
\end{align*}
which proves the lemma.
\end{proof}

\begin{lem}\label{lem:hitup}
Assume that $ H_0=0$. Then
for $0<a_1<a_2$,
\begin{itemize}
\item[(i)] If $a_1>g$, then for $t\ge 2(a_2-a_1)/g$,
\begin{equation}\label{equation:hitup}
\Prob(\vtau_{-g+a_1}>t \mid V_0=-g+a_2) \le e^{-g(a_1-g)t}.
\end{equation}
\item[(ii)] If $a_1 \le g$, then for $t\ge 2(a_2-a_1)/a_1$,
\begin{equation}\label{equation:hituptwo}
\Prob(\vtau_{-g+a_1}>t \mid V_0=-g+a_2) \le \frac{4}{a_1\sqrt{2\pi t}}e^{-a_1^2t/8}.
\end{equation}
\end{itemize}
\end{lem}

\begin{proof}
For $t<\vtau_{-g+a_1}$,  the following inequality holds, $\displaystyle{L_t \le L^{(-g+a_1)}_t=\sup_{u\le t}(B_u-(a_1-g)u)}$. Therefore for $a_1>g$ and $t\ge 2(a_2-a_1)/g$, 
\begin{eqnarray*}
\Prob(\vtau_{-g+a_1}>t \mid V_0=-g+a_2)
&=&\Prob(L_s-gs \geq -(a_2-a_1) \text{ for } s\leq t)\\
&\le&\Prob\left(\sup_{u\le s}(B_u-(a_1-g)u)-gs \geq -(a_2-a_1) 
\text{ for } s\leq t\right)\\
&\le&\Prob\left(\sup_{u\le t}(B_u-(a_1-g)u)-gt \geq -(a_2-a_1) \right)\\
&\le&\Prob\left(\sup_{u < \infty}(B_u-(a_1-g)u) > gt-(a_2-a_1)\right) \\
&=& \Prob\left( \tau^{B,-(a_1-g)}_{gt-(a_2-a_1)} < \infty \right). 
\end{eqnarray*}
We use \eqref{oc13.2} and the assumption that $a_2-a_1 \le gt/2$
to conclude that
\begin{align*}
\Prob(\vtau_{-g+a_1}>t \mid V_0=-g+a_2)
&\leq
\exp( - 2 (a_1-g)(gt - (a_2 - a_1))
\leq e^{-g(a_1-g)t}.
\end{align*}
This proves $(i)$.

For $a_1 \le g$ and $t\ge 2(a_2-a_1)/a_1$,
\begin{align*}
\Prob(\vtau_{-g+a_1}>t \mid V_0=-g+a_2) 
&=\Prob(L_s-gs \geq -(a_2-a_1) \text{ for } s\leq t)\\
& \le \Prob\left(\sup_{u\le s}(B_u+(g-a_1)u)-gs \geq -(a_2-a_1) \text{ for } 
s\leq t\right)\\
&\le \Prob\left(\sup_{u\le t}B_u+(g-a_1)t-gt >-(a_2-a_1)\right)\\
& \le \Prob\left(\sup_{u \le t}B_u > a_1t/2\right).
\end{align*}
This and \eqref{oc13.3} show that
\begin{align*}
\Prob(\vtau_{-g+a_1}>t \mid V_0=-g+a_2) 
 \le \frac{4}{a_1\sqrt{2\pi t}}e^{-a_1^2t/8},
\end{align*}
which proves $(ii)$.
\end{proof}

The following lemma gives a uniform control over $\sigma(t)$ (the first time the Brownian particle and the inert particle meet after time $t$) over all times in a large interval.
\begin{lem}\label{lem:unifcontrol}
For every $\delta >0$ and $C_0$, we can find positive constants $C_1,C_2, a_0$ such that for all $a \ge a_0$ and all $v \in \left[-g-\delta^2 a^2/8, -g+\sqrt{g}\delta a/4\right]$, the following holds for $V_0=v$, $H_0=0$, and any $m \ge 1$:
\begin{equation*}
\Prob\left(\sigma(t) > t+3\frac{\delta a}{\sqrt{g}} \text{ for some } 
t \le C_0 a^m \wedge \vtau_{-g+\sqrt{g}\delta a/4}\wedge \vtau_{-g-\delta^2 a^2/8}\right) \le C_1a^{m-1} e^{-C_2a^3}.
\end{equation*}
\end{lem}

\begin{proof}
Fix any $\delta>0$ and $v \in [-g-\frac{\delta^2 a^2}{8}, -g+\frac{\sqrt{g}\delta a}{4}]$. Assume that $V_0=v$ and $H_0=0$.
 Note that for $t \le \vtau_{-g+\sqrt{g}\delta a/4}$,
\begin{eqnarray*}
L_t-gt &\le& \frac{\delta^2 a^2}{8}+ \frac{\sqrt{g}\delta a}{4},\\
S_t &\le& \left(-g+\frac{\sqrt{g}\delta a}{4}\right)t.
\end{eqnarray*}
These together yield 
\begin{align*}
S_t-X_t 
= S_t -B_t +L_t
\le -B_t +\frac{\sqrt{g}\delta a}{4}t +\frac{\delta^2 a^2}{8}+\frac{\sqrt{g}\delta a}{4}.
\end{align*}
Thus we have
\begin{align}\label{align:diffboundone}
\Prob\left(S_{\delta a/\sqrt{g}}-X_{\delta a/\sqrt{g}} > \delta^2a^2,\, \frac{\delta a}{\sqrt{g}} \le \vtau_{-g+\sqrt{g}\delta a/4} \right)
\nonumber
 &\le \Prob\left(-B_{\delta a/\sqrt{g}}+ \frac{3\delta^2a^2}{8}+\frac{\sqrt{g}\delta a}{4} > \delta^2a^2\right)\\
&\leq
\Prob\left(-B_{\delta a/\sqrt{g}} > \frac 9 {16} \delta^2a^2\right)
\le e^{-\sqrt{g}\delta^3a^3/8}.
\end{align}
The second inequality holds for large enough $a$ (depending on $\delta$) and the last inequality follows from \eqref{oc13.3}. 

Suppose that the following event holds
\begin{align}\label{oc15.2}
\left\{\sigma\left(\frac{\delta a}{\sqrt{g}}\right) > 3\frac{\delta a}{\sqrt{g}}, 
S_{\delta a/\sqrt{g}}-X_{\delta a/\sqrt{g}} \le \delta^2a^2, 
\frac{\delta a}{\sqrt{g}} \le \vtau_{-g+\sqrt{g}\delta a/4}
\right\}.
\end{align}
Then $S$ is a parabola on the interval 
$\left[\frac{\delta a}{\sqrt{g}},  \frac{3\delta a}{\sqrt{g}}\right]$. 
If $w = V_{\delta a/\sqrt{g}}$ then the parabola increment  over this interval is
\begin{align*}
- \frac 1 2 g \left(\frac{2\delta a}{\sqrt{g}}\right)^2 + w \frac{\delta a}{\sqrt{g}}
\leq
- \frac 1 2 g \left(\frac{2\delta a}{\sqrt{g}}\right)^2 
+ \left(-g+\frac{\sqrt{g}\delta a}{4}\right) \frac{\delta a}{\sqrt{g}}
\leq 
- \frac 3 2 \delta^2 a^2.
\end{align*}
Since $S_{\delta a/\sqrt{g}}-X_{\delta a/\sqrt{g}} \le \delta^2a^2$ and
$X$ stays below $S$ on the interval $\left[\frac{\delta a}{\sqrt{g}},  \frac{3\delta a}{\sqrt{g}}\right]$, the following event must hold,
\begin{align*}
\left\{B_{3\delta a/\sqrt{g}} - B_{\delta a/\sqrt{g}}\le -\delta^2a^2/2\right\}.
\end{align*}
Recalling \eqref{oc15.2}, we conclude that
\begin{align*}
\Prob\left(\sigma\left(\frac{\delta a}{\sqrt{g}}\right) > 3\frac{\delta a}{\sqrt{g}}, 
S_{\delta a/\sqrt{g}}-X_{\delta a/\sqrt{g}} \le \delta^2a^2, 
\frac{\delta a}{\sqrt{g}} \le \vtau_{-g+\sqrt{g}\delta a/4}
\right)
\leq
\Prob\left(B_{\frac{3\delta a}{\sqrt{g}}} - B_{\frac{\delta a}{\sqrt{g}}}\le -\delta^2a^2/2\right).
\end{align*}
This, \eqref{align:diffboundone} and \eqref{oc13.3} yield for large $a$,
\begin{align}\label{align:diffboundtwo}
&\Prob\Big(\sigma\left(\frac{\delta a}{\sqrt{g}}\right)> 3\frac{\delta a}{\sqrt{g}}, \ \frac{\delta a}{\sqrt{g}} \le \vtau_{-g+\sqrt{g}\delta a/4}
\Big)\nonumber\\
& \le \Prob\left(S_{\delta a/\sqrt{g}}-X_{\delta a/\sqrt{g}} > \delta^2a^2, \frac{\delta a}{\sqrt{g}} \le \vtau_{-g+\sqrt{g}\delta a/4}
\right)\nonumber\\
&\quad + \Prob\left(\sigma\left(\frac{\delta a}{\sqrt{g}}\right) > 3\frac{\delta a}{\sqrt{g}}, S_{\delta a/\sqrt{g}}-X_{\delta a/\sqrt{g}} \le \delta^2a^2, \frac{\delta a}{\sqrt{g}} \le \vtau_{-g+\sqrt{g}\delta a/4}\right)\nonumber\\
& \le e^{-\sqrt{g}\delta^3a^3/8} + \Prob\left(B_{3\delta a/\sqrt{g}} 
- B_{\delta a/\sqrt{g}}\le -\delta^2a^2/2\right)\nonumber\\
& \le  e^{-\sqrt{g}\delta^3a^3/8} + e^{-\sqrt{g}\delta^3a^3/16}.
\end{align}

Define  stopping times $\tgam_0=0$ and 
$$\tgam_{k+1}=\inf\{t \ge \tgam_k+ \delta a/\sqrt{g}: H_t=0\},$$ 
 for $k \ge 0$. Then by \eqref{align:diffboundtwo}, the strong Markov property applied at $\tgam_k$, and Remark \ref{oc18.1}, for large $a$,
\begin{align}\label{align:diffboundthree}
&\Prob\left(\tgam_{k+1}-\tgam_k  > 3\frac{\delta a}{\sqrt{g}},
 \ \vtau_{-g+\sqrt{g}\delta a/4} \ge  \tgam_k+ \delta a/\sqrt{g},
 \ \vtau_{-g-\delta^2 a^2/8} \ge \tgam_k\right)\nonumber\\
&\leq \Prob\left(\tgam_{k+1}-\tgam_k  > 3\frac{\delta a}{\sqrt{g}}, \ \vtau_{-g+\sqrt{g}\delta a/4} \ge \tgam_k+ \delta a/\sqrt{g}\right)\nonumber\\
& \le e^{-\sqrt{g}\delta^3a^3/8} + e^{-\sqrt{g}\delta^3a^3/16}.
\end{align}

Consider any $m \ge 1$. Suppose that there is $t_1 \in \left[0,C_0 a^m \wedge \vtau_{-g+\sqrt{g}\delta a/4} \wedge \vtau_{-g-\delta^2 a^2/8}\right]$ such that $\sigma(t_1) > t_1+3\frac{\delta a}{\sqrt{g}}$. Then $t_1 \in \left[0,C_0 a^m\right]$ and, therefore,  we can find $0 \le k_1 \le 
C_0 \sqrt{g}a^{m-1}/\delta$ with $\tgam_{k_1} \le t_1 \le \tgam_{k_1+1}$, because $\tgam_{k+1}- \tgam_k\geq \delta a/\sqrt{g}$ for all $k$. 

It follows from the definition of $\tgam_{k+1}$ that $\sigma(t_1)-t_1 \le \tgam_{k_1+1}-\tgam_{k_1}$. 
Since $V_0 \leq -g+\frac{\sqrt{g}\delta a}{4}$, the processes $S$ and $X$ must take the same value at the time  $\vtau_{-g+\sqrt{g}\delta a/4}$, by Remark \ref{oc16.1}. Hence, if 
$\vtau_{-g+\sqrt{g}\delta a/4} \in [t_1, \tgam_{k_1}+\delta a/\sqrt{g}]$, then
$$
\sigma(t_1) \le \vtau_{-g+\sqrt{g}\delta a/4}  \le \tgam_{k_1} +\delta a/\sqrt{g} \leq t_1 + \delta a/\sqrt{g},
$$
which contradicts the assumption that $\sigma(t_1) > t_1+3\frac{\delta a}{\sqrt{g}}$. 
Thus, if $t_1 \leq \vtau_{-g+\sqrt{g}\delta a/4}$ 
and $\sigma(t_1) > t_1+3\frac{\delta a}{\sqrt{g}}$ then
$\vtau_{-g+\sqrt{g}\delta a/4} \geq \tgam_{k_1} +\delta a/\sqrt{g}$.
These observations and \eqref{align:diffboundthree} imply that
\begin{align*}
&\Prob\left(\sigma(t)  > t+3\frac{\delta a}{\sqrt{g}} \text{ for some } t \le a^m \wedge \vtau_{-g+\sqrt{g}\delta a/4} \wedge \vtau_{-g-\delta^2 a^2 /8}\right)\nonumber\\
&\le \sum_{k=0}^{\left\lfloor C_0\sqrt{g}a^{m-1}/\delta\right\rfloor}\Prob\left(\sigma(t) > t+3\frac{\delta a}{\sqrt{g}} , 
t \le  \vtau_{-g+\sqrt{g}\delta a/4} \wedge \vtau_{-g-\delta^2 a^2/8}
\text{ for some } t \in [\tgam_k, \tgam_{k+1}]
\right)\\
&\le \sum_{k=0}^{\left\lfloor C_0\sqrt{g}a^{m-1}/\delta\right\rfloor}\Prob\left(\tgam_{k+1}-\tgam_k > 3\frac{\delta a}{\sqrt{g}}, 
\ \vtau_{-g+\sqrt{g}\delta a/4} \ge 
\tgam_k +\delta a/\sqrt{g},
 \ \vtau_{-g-\delta^2 a^2/8} \ge \tgam_k\right)\\
& \le \left(\left\lfloor C_0\sqrt{g}a^{m-1}/\delta\right\rfloor +1\right)\left(e^{-\sqrt{g}\delta^3a^3/8} + e^{-\sqrt{g}\delta^3a^3 /16}\right).
\end{align*}
 This proves the lemma.
\end{proof}

The following lemma tells us that the probability of the velocity staying inside the interval $[-g-a,-g+a]$ for all times up to $t$ decays exponentially with $t$.

\begin{lem}\label{lem:tube}
For any $a>0$, there exists $p_0 \in (0,1)$ depending on $a$ such that for any integer $m \ge 1$,
\begin{equation}\label{equation:tube}
\sup_{v \in [-g-a,-g+a]}\Prob\left(|V_t+g| \le a \text{ for all } t \in \left[0,m\left(1+2a/g\right)\right]\mid V_0=v, H_0=0\right) \le p_0^m.
\end{equation}
\end{lem}
\begin{proof}
Assume without loss of generality that $B_0=0$.
Note that $L_t=\sup_{u\le t}(B_u-S_u) \ge B_t-S_t$. If $V_u \in [-g-a,-g+a]$ for all $0\le u \le t$, then $S_t \le (-g+a)t$. This gives $L_t \ge B_t + (g-a)t$. Thus,
$$V_t=V_0 +L_t-gt \ge -g-a +B_t +(g-a)t -gt = B_t -at-g-a,$$
and, therefore,
\begin{align*}
\{|V_t+g|  \le a \text{ for all } t \in [0,1]\}
\subset
\{V_1 + g \leq a\}
\subset
\{B_1 - 2 a \leq a\}
=\{B_1 \leq 3a\}.
\end{align*}
Let $p_0 = \Prob(B_1 \leq 3a ) < 1$.
Then,
\begin{align*}
\sup_{v \in [-g-a,-g+a]}\Prob(|V_t+g| & \le a \text{ for all } t \in [0,1]\mid V_0=v, H_0=0)\\
&\le \Prob(B_1\le 3a)=p_0 .
\end{align*}
If $V_u \in [-g-a,-g+a]$ for all $0\le u \le 1+3a/g$, then $\sigma(1) \le  1+2a/g$, 
for otherwise there would be $t \in [1+2a/g, 1+3a/g)$ with $V_t<-g-a$. Thus, applying the strong Markov property at $\sigma(1)$, in view of Remark \ref{oc18.1}, we have for any $m \ge 2$,
\begin{align*}
&\sup_{v \in [-g-a,-g+a]}\Prob\left(|V_t+g| \le a \text{ for all } t \in \left[0,m\left(1+2a/g\right)\right]\mid V_0=v, H_0=0\right)\\
&\quad \le \sup_{v \in [-g-a,-g+a]}\Prob(|V_t+g| \le a \text{ for all } t \in [0,1]\mid V_0=v, H_0=0)\\
& \quad \quad \ \ \ \times \sup_{v \in [-g-a,-g+a]}\Prob\left(|V_t+g| \le a \text{ for all } t \in \left[0,(m-1)\left(1+2a/g\right)\right]\mid V_0=v, H_0=0\right)\\
& \quad \le p_0 \times \sup_{v \in [-g-a,-g+a]}\Prob\left(|V_t+g| \le a \text{ for all } t \in \left[0,(m-1)\left(1+2a/g\right)\right]\mid V_0=v, H_0=0\right).
\end{align*}
Recursively applying the same argument, we get \eqref{equation:tube}.
\end{proof}

\section{Fluctuations {for excursions of} $V$ and $S-X$ }\label{sec:fluct}

For large positive values of $V$, the inert particle gets pulled down by gravitation which can be viewed as a \textit{soft potential}. For large negative $V$, the inert particle gets pushed up by \textit{hard reflection} forces (`hard' because the support of the intersection local time has Lebesgue measure zero). These two kinds of forces are very different in nature and it is natural to expect that the effect each exerts on the inert particle is also different. This is formalized in Theorems \ref{thm:fluconeside} and \ref{thm:gaponeside} below. In Theorem \ref{thm:fluconeside}, we prove that the distribution of excursions of the velocity process $V$ above the level $-g$ has a different tail behavior compared to that of excursions below $-g$. In Theorem \ref{thm:gaponeside}, we prove that the probability of having a large gap $S-X$ during an excursion of $V$ above $-g$ behaves differently than that of having a large gap during an excursion of $V$ below $-g$.

These theorems, besides being of independent interest, will be essential in deriving fluctuation results for $V$ and $S-X$ between certain renewal times defined in Section \ref{sec:renewal} {
and, consequently, in proving Theorems \ref{thm:velfluc} and  \ref{thm:stronglaw}.}

\begin{thm}\label{thm:fluconeside}
There exist positive constants $a_0$ and $C^+_i,C^-_i$ for $i=1,2,3,4$ such that for any $a\ge a_0$:
\begin{equation}\label{equation:lowerfluc}
C^-_1e^{-C^-_2a^3} \le \Prob(\vtau_{-g-a} < \vtau_{-g-1} \mid V_0=v, H_0=0) \le C^-_3e^{-C^-_4a^3}
\end{equation}
uniformly over all $v \in [-g-a/2, -g-2]$,
and
\begin{equation}\label{equation:upperfluc}
C^+_1e^{-C^+_2a^2} \le \Prob(\vtau_{-g+a} < \vtau_{-g+1} \mid V_0=v, H_0=0) \le C^+_3e^{-C^+_4a^2}
\end{equation}
uniformly over all $v \in [-g+2, -g+a/2]$.
\end{thm}
\begin{proof}
Recall that $C, C', C'', \dots$ denote positive constants whose values might change from line to line.
The initial conditions will be suppressed in the notation most of the time. We will assume that  $V_0=v$ and $H_0=0$ unless stated otherwise.

First we will prove the upper bound in \eqref{equation:lowerfluc}.
Let $\delta = 1/(24\sqrt{g})$ and 
$$
F_a=\displaystyle{\left\lbrace\sigma(t) \le t+3\frac{\delta a}{\sqrt{g}} \text{ for all } t \le \vtau_{-g-1} \wedge \vtau_{-g-a}\right\rbrace}.
$$ 
Suppose that $a$ is large enough so that $\frac{\delta^2 a^2}{8} > a$.
Then
the assumption $v \in [-g-a/2, -g-2]$ made in \eqref{equation:lowerfluc}
implies that
$v \in \left[-g-\frac{\delta^2 a^2}{8}, -g+\frac{\sqrt{g}\delta a}{4}\right]$,
an assumption of Lemma \ref{lem:unifcontrol}. Moreover
\begin{align*}
F_a^c \cap \{ \vtau_{-g-1} \le a^3\}
&\subset
\left\lbrace\sigma(t) > t+3\frac{\delta a}{\sqrt{g}} 
\text{ for some } t \le a^3 \land \vtau_{-g-1} \wedge \vtau_{-g-a}\right\rbrace \\
&\subset
\left\{\sigma(t) > t+3\frac{\delta a}{\sqrt{g}} \text{ for some } t \le a^3 \wedge \vtau_{-g+\sqrt{g}\delta a/4}\wedge \vtau_{-g-\delta^2 a^2/8}\right\}.
\end{align*}
We now use Lemma \ref{lem:unifcontrol} to see that
\begin{align}\label{oc18.3}
\Prob(F_a^c, \vtau_{-g-1} \le a^3)
\leq  C\frac{a^2}{\delta} e^{-C' a^3}.
\end{align}

We apply Lemma  \ref{lem:hittimeest} with
$a_1=1$ and $a_2 = -g -v$. 
The assumption of the lemma that $0 < a_1 < a_2$ is satisfied because we assume that $v< -g-2$ in \eqref{equation:lowerfluc}.
The lemma implies that for 
$$a^3 \geq 2(a_2 - a_1)/a_1 = 2(-g-v -1)/1 \geq 2(-g-(-g-2)-1) =2,$$
we have
\begin{align}
\Prob(\vtau_{-g-1}>a^3) 
&\leq \frac{4 (-g-v -1) }{((-g-v-1) + a^3)\sqrt{2\pi a^3}} e^{-a^3 /8}
\leq
\frac{4 (-g-(-g-2) -1) }{((-g-(-g-2)-1) + a^3)\sqrt{2\pi a^3}} e^{-a^3 /8}
\nonumber\\
&\leq Ce^{-C'a^3}. \label{oc18.2}
\end{align}
We combine \eqref{oc18.3} and \eqref{oc18.2} to obtain for large enough $a$,
\begin{eqnarray}\label{eqnarray:diffboundfour}
\Prob(\vtau_{-g-a} < \vtau_{-g-1}) 
& \le & \Prob(\vtau_{-g-a} < \vtau_{-g-1}, F_a )+\Prob(F_a^c, \vtau_{-g-1} 
\le a^3)+ \Prob(\vtau_{-g-1}>a^3)\nonumber\\
& \le & \Prob(\vtau_{-g-a} < \vtau_{-g-1}, F_a ) + C\frac{a^2}{\delta} e^{-C'a^3} + C''e^{-C'''a^3}.
\end{eqnarray}

Next we will estimate the  probability 
on the right hand side of \eqref{eqnarray:diffboundfour}. Note that if $\{\vtau_{-g-a} < \vtau_{-g-1}\} \cap F_a$ holds, then $\displaystyle{\sigma\left(\vtau_{-g-a/2}\right) \le \vtau_{-g-a/2}+ 3\frac{\delta a}{\sqrt{g}}} < \vtau_{-g-a/2}+\frac{a}{4g}$. 
The maximum rate of decrease for velocity $V$ is $g$, so $V$ can decrease by at most $a/4$ over the interval $\left[\vtau_{-g-a/2}, \sigma\left(\vtau_{-g-a/2}\right)\right]$, and, therefore, 
\begin{align}\label{oc19.1}
V\left(\sigma\left(\vtau_{-g-a/2}\right)\right) \geq -g-3a/4.
\end{align}
Thus, applying the strong Markov property at the stopping time $\sigma\left(\vtau_{-g-a/2}\right)$ and Remark \ref{oc18.1}, we obtain for $v \in [-g-a/2, -g-2]$,
\begin{equation}\label{equation:ineqone}
\Prob(\vtau_{-g-a} < \vtau_{-g-1}, F_a \mid V_0=v, H_0=0) 
\le \sup_{v \in [-g-3a/4,-g-a/2]}\Prob(\vtau_{-g-a} < \vtau_{-g-1}, F_a ).
\end{equation}
Suppose that $v \in  [-g-3a/4,-g-a/2]$.
If the event $\{\vtau_{-g-a} < \vtau_{-g-a/4}\}$ holds, then for any $t \in [0,a/(4g))$, as  $V_t \ge v -gt > -g-a$, we have $t < \vtau_{-g-a} < \vtau_{-g-a/4}$. Therefore, 
\begin{align*}
L_t-gt = V_t - V_0 \leq -g-a/4 -(-g-3a/4) =  a/2.
\end{align*}
Since $t \le \vtau_{-g-a/4}$ implies $L_t \ge L^{(-g-a/4)}_t=\sup_{u \le t}(B_u +(g+a/4)u)$, we obtain that
\begin{align*}
B_t + (g+a/4)t -gt \leq a/2
\end{align*}
for all $t \in [0,a/(4g)]$. Thus, for any $v \in  [-g-3a/4,-g-a/2]$,
using \eqref{oc13.3}, for large enough $a$,
\begin{align}\label{eqnarray:lowerlevelbound}
\Prob\left(\vtau_{-g-a} < \vtau_{-g-a/4} \right) 
\le  \Prob\left(B_{a/(4g)} \le a/2-a^2/(16g)\right) \le Ce^{-C'a^3}.
\end{align}

We define a sequence of stopping times by setting $T_0=0$, and for $k \ge 0$, 
\begin{align*}
T_{2k+1}&=
\begin{cases}
\inf\{t \ge T_{2k}:V_t=-g-a \text{ or } V_t=-g-a/4\} & \mbox{if } V_{T_{2k}} \neq -g-1, -g-a,\\
T_{2k} &\text{otherwise, }
\end{cases}\\
T_{2k+2}&=
\begin{cases}
\inf\{t \ge T_{2k+1}:V_t=-g-1 \text{ or } V_t=-g-a/2\} & \mbox{if } V_{T_{2k+1}} \neq -g-1, -g-a,\\
T_{2k+1} &\text{otherwise. }
\end{cases}
\end{align*}
Let
$$p_0(v,a)=\Prob(V_{T_1}=-g-a, F_a  \mid V_0=v, H_0=0),$$
and for $k \ge 1$, 
\begin{align}\label{eqnarray:termtwo}
&p_k(v,a)\\
&=\Prob(V_{T_{2j+1}}=-g-a/4, V_{T_{2j+2}}=-g-a/2 \text{ for }0 \le j <k,
 V_{T_{2k+1}}=-g-a, F_a  \mid V_0=v, H_0=0).\nonumber
\end{align}
We can then decompose the velocity path between the stopping times $\{T_k\}_{k \ge 1}$ to obtain
\begin{align}\label{eqnarray:pathdec}
\Prob(\vtau_{-g-a} < \vtau_{-g-1}, F_a  \mid V_0=v, H_0=0)=\sum_{k=0}^{\infty}p_k(v,a).
\end{align}
By \eqref{eqnarray:lowerlevelbound},
\begin{equation}\label{equation:pzero}
\sup_{v \in [-g-3a/4, -g-a/2]}p_0(v,a) \le Ce^{-C'a^3}.
\end{equation}
If the events $\{V_{T_2}=-g-a/2\}$ and $F_a$ hold then the event 
$\{V_{\sigma(T_2)} \in [-g-3a/4, -g-a/2]\}$ holds as well---this claim can be
proved just like \eqref{oc19.1}.
This and the strong Markov property applied at $T_1$ show that for $k \ge 1$,
and $v \in [-g-3a/4, -g-a/2]$,
\begin{align}\label{eqnarray:recrel}
&p_k(v,a)\\
  &\le \Expect\left(\Indicator_{\kpar{-g-a/2}}(V_{T_2}) \Indicator_{ [-g-3a/4, -g-a/2]}(V_{\sigma(T_2)} ) p_{k-1}(V_{\sigma(T_2)},a) \mid V_0=-g-a/4, H_0=0\right)\nonumber\\
&\le \Prob\left(V_{T_2}=-g-a/2\mid V_0=-g-a/4, H_0=0\right)
 \sup_{v \in [-g-3a/4, -g-a/2]}p_{k-1}(v,a).\nonumber
\end{align}

If 
$V_0=-g-a/4$ and $H_0=0$, then for $t \le \vtau_{-g-1}$, we have $L_t \ge L^{(-g-1)}_t$. Thus,
$$V_t-V_0 = L_t -gt \ge \sup_{u \le t}(B_u +(g+1)u)-gt \ge B_t +t.$$
This and the fact that the scale function for $B_t+t $ is $e^{-2x}$ (see \cite[Example~15.4.B]{KT})
imply that 
\begin{align*}
\Prob\left(V_{T_2}=-g-a/2\mid V_0=-g-a/4, H_0=0\right) 
  \le  \Prob(  \tau^{B,1}_{-a/4}  \leq \tau^{B,1}_{a/4-1} )   
=\frac{1-e^{-(a/2-2)}}{e^{a/2}-e^{-(a/2-2)}} \le C e^{-a/2},
\end{align*}
for sufficiently large $a$. Combining this estimate with \eqref{equation:pzero} and \eqref{eqnarray:recrel}, we obtain
$$\sup_{v \in [-g-3a/4, -g-a/2]}p_k(v,a) \le (Ce^{-a/2})^kC'e^{-C''a^3}.$$
This in turn can be substituted  into \eqref{eqnarray:pathdec} to show that
$$\sup_{v \in [-g-3a/4, -g-a/2]}\Prob(\vtau_{-g-a} < \vtau_{-g-1}, F_a  \mid V_0=v, H_0=0) \le Ce^{-C'a^3}.$$
This,  \eqref{eqnarray:diffboundfour} and \eqref{equation:ineqone} yield the upper bound in \eqref{equation:lowerfluc}.

\medskip

Next we will prove the lower bound in \eqref{equation:lowerfluc}. First, we will
argue that it suffices to work with the case $V_0=-g-2, H_0=0$. This is because, by the strong Markov property applied at $\vtau_{-g-2}$
and Remarks \ref{oc18.1} and \ref{oc16.1}, we can write for any $v \in [-g-a/2, -g-2]$,
\begin{align*}
\Prob(\vtau_{-g-a}<\vtau_{-g-1} \mid V_0=v, H_0=0) 
& \ge \Prob(\vtau_{-g-2}<\vtau_{-g-a}\mid V_0=v, H_0=0)\\
& \quad \times\Prob(\vtau_{-g-a}<\vtau_{-g-1}\mid V_0=-g-2, H_0=0).
\end{align*}
From the upper bound in \eqref{equation:lowerfluc}, it follows that for large enough $a$,
\begin{align*}
\inf_{v\in [-g-a/2, -g-2]}\Prob(\vtau_{-g-2}<\vtau_{-g-a} & \mid V_0=v, H_0=0)\\
&\ge \inf_{v\in [-g-a/2, -g-2]}\Prob(\vtau_{-g-1}<\vtau_{-g-a} \mid V_0=v, H_0=0) \ge \frac{1}{2}.
\end{align*}
Thus, for large enough $a$,
\begin{align*}
\inf_{v\in [-g-a/2, -g-2]}\Prob(\vtau_{-g-a}<\vtau_{-g-1} & \mid V_0=v, H_0=0)\\
& \ge \frac{1}{2}\times\Prob(\vtau_{-g-a}<\vtau_{-g-1}\mid V_0=-g-2, H_0=0),
\end{align*}
which proves our claim that it is enough to consider
the case $V_0=-g-2, H_0=0$.

Fix $\eps \in (0,1/g)$ such that $\displaystyle{(g+2)\eps+\frac{g\eps^2}{2}<\frac{1}{2}}$. As the local time $L_t$ is non-negative, $V_t-V_0 \ge -gt$ for all $t \ge 0$. Thus, assuming $V_0=-(g+2)$ and $H_0=0$, we obtain $S_t \ge -(g+2)t-gt^2/2$ for all $t \ge 0$. Note that the event
\begin{equation}\label{equation:eventdef}
A=\left\lbrace\sup_{u \le \eps}(B_u) < 1-(g+2)\eps-g\eps^2/2, B_{\eps}<-\eps-(g+2)\eps-g\eps^2/2\right\rbrace
\end{equation}
occurs with positive probability, say, $q$. If $A$ happens 
 then for any $t \in [0,\eps]$,
\begin{eqnarray*}
V_t-V_0=L_t-gt \le L_t = \sup_{u\le t}(B_u-S_u) &\le& \sup_{u \le t}(B_u)+\sup_{u \le t}(-S_u)\\
&\le& \sup_{u \le \eps}(B_u)+(g+2)\eps+ g\eps^2/2 < 1.
\end{eqnarray*}
This implies that if $V_0 = -g-2$ then $\vtau_{-g-1}>\eps$. Furthermore,
\begin{eqnarray*}
S_{\eps}-X_{\eps}=S_{\eps}-B_{\eps}+L_{\eps} \ge S_{\eps}-B_{\eps} \ge -(g+2)\eps- g\eps^2/2-B_{\eps}>\eps.
\end{eqnarray*}
Thus, 
\begin{align}\label{oc27.1}
A \subset\{\vtau_{-g-1}>\eps, S_{\eps}-X_{\eps}>\eps\}.
\end{align}
Hence, assuming that $V_0=-g-2$ and $H_0=0$, we will argue that the following series of inequalities holds:
\begin{eqnarray}\label{eqnarray:lowerbd}
\Prob(\vtau_{-g-a}<\vtau_{-g-1}) & \ge & \Prob(\vtau_{-g-1}>\eps, S_{\eps}-X_{\eps}>\eps, X_t<S_t \text{ for all } t\in [\eps,a/g])\nonumber\\
&\ge& \Prob(\vtau_{-g-1}>\eps, S_{\eps}-X_{\eps}>\eps)\nonumber\\
&\quad & \times\inf_{v \ge -(g+3), z \ge \eps}\Prob(X_t<S_t \text{ for all } t\in [0,a/g] \mid V_0=v,H_0=z)\nonumber\\
&\ge& q \Prob(X_t<S_t \text{ for all } t\in [0,a/g] \mid V_0=-(g+3),H_0=\eps).
\end{eqnarray}
The first inequality above holds because if $\vtau_{-g-1}>\eps$ and $X_t<S_t$ for all $t \in [\eps, a/g]$, then
the velocity $V_t$ decreases linearly and reaches the level $-g-a$ for some $t \in [\eps, a/g]$. 
The second inequality above follows from the Markov property 
applied at time $\eps$ combined with the fact that $V_{\eps} \ge -(g+2)-g\eps > -(g+3)$ as $\eps <1/g$. To see the third inequality, first note that $\Prob(\vtau_{-g-1}>\eps, S_{\eps}-X_{\eps}>\eps) \ge \Prob(A)=q$. 
If $X_t<S_t$ for all $ t\in [0,a/g]$ then $S$ is a parabola on this 
time interval. 
For $v \ge -(g+3)$ and $z \ge \eps$,
the parabola $t \mapsto X_0 +z+vt-\frac{gt^2}{2}$
lies above the parabola
$t \mapsto X_0 +\eps - (g+3)t-\frac{gt^2}{2}$
or the two parabolas are identical.
This implies the last inequality in \eqref{eqnarray:lowerbd}.

Observe that
\begin{align*}
\Prob&(X_t<S_t \text{ for all } t\in [0,a/g] \mid V_0=-(g+3),H_0=\eps)\\
&=\Prob(\eps-(g+3)t-gt^2/2 - B_t > 0 \text{ for } t\leq a/g)\\
&=\Prob(\eps/g-(g+3)t/g-t^2/2 - (1/g)B_t > 0 \text{ for } t\leq a/g).
\end{align*}
The form of the expression on the last line is needed so that it matches 
the notation in \cite{hofstad}. 
To use an estimate from that paper, we introduce the following notation.
\begin{align*}
R(t) = (t + (g+3)/g)^3 - t^3/4 - (g+3)^3/g^3 - 6 ((g+3)/g) (\eps/g)
- 3 (\eps/g) t.
\end{align*}
We will write $R'(t) = \frac\prt{\prt t} R(t)$.
By \cite[Cor.~2.1]{hofstad}, we obtain for large enough $a$,
\begin{align*}
\Prob&(X_t<S_t \text{ for all } t\in [0,a/g] \mid V_0=-(g+3),H_0=\eps)\\
&=\Prob(\eps/g-(g+3)t/g-t^2/2 - (1/g)B_t > 0 \text{ for } t\leq a/g)\\
&=\frac{ 3 (1/g) \sqrt{a/g}}{\sqrt{2\pi}}
\frac{ \exp( - R(a/g) / (6 /g^2)) }{ R'(a/g)}
(1 + O (g/a))
\ge Ce^{-C'a^3},
\end{align*}
which, along with \eqref{eqnarray:lowerbd}, yields the lower bound in \eqref{equation:lowerfluc}.

\medskip

Next, we set out to prove the upper bound in \eqref{equation:upperfluc}. For $a\ge 4$, applying the strong Markov property at time $\vtau_{-g+a-1}$ and
Remarks \ref{oc18.1} and \ref{oc16.1}, we have for any $v \in [-g+2,-g+a-1]$,
\begin{align}\label{eqnarray:markov}
&\Prob(\vtau_{-g+a} < \vtau_{-g+1} \mid V_0=v, H_0=0)\\
&= \Prob(\vtau_{-g+a-1} < \vtau_{-g+1} \mid V_0=v, H_0=0)
 \Prob(\vtau_{-g+a} < \vtau_{-g+1} \mid V_0=-g+a-1, H_0=0).\nonumber
\end{align}
We  estimate the second probability on the right hand side of the above equation as follows,
\begin{align}\label{eqnarray:probdec}
&\Prob(\vtau_{-g+a} < \vtau_{-g+1} \mid V_0=-g+a-1, H_0=0)\\
&=\sum_{k=2}^{a-1}\Prob(\vtau_{-g+k} < \vtau_{-g+a}\le \vtau_{-g+k-1} \mid V_0=-g+a-1, H_0=0).\nonumber
\end{align}
If $V_0 = -g+k$ and
 $t \le \vtau_{-g+k-1}$ then $L_t \le L^{(-g+k-1)}_t=\sup_{u \le t}(B_u + (g-k+1)u)$. Thus, for $2 \le k \le a-1$,
\begin{align*}
&\Prob(\vtau_{-g+k} < \vtau_{-g+a}\le \vtau_{-g+k-1} \mid V_0=-g+a-1, H_0=0)\\
 &\le \Prob(\vtau_{-g+a}\le \vtau_{-g+k-1} \mid V_0=-g+k, H_0=0)\\
 &= \Prob\left(\inf\{t\geq 0: L_t-gt =a-k\} < \inf\{t\geq 0: L_t-gt =-1\}   \mid V_0=-g+k, H_0=0\right)\\
 &\le\Prob\big(
\inf\{t\geq 0: L^{(-g+k-1)}_t-gt =a-k\} \\
&\qquad \qquad < \inf\{t\geq 0: L^{(-g+k-1)}_t-gt =-1\} \mid V_0=-g+k, H_0=0\big)\\
 &\le\Prob\left(\inf\{t\geq 0: L^{(-g+k-1)}_t-gt =a-k\} < \infty \mid V_0=-g+k, H_0=0\right)\\
 &=\Prob\left(\inf\{t\geq 0: B_t + (g-k+1)t-gt =a-k\} < \infty
\right)\\
 &=\Prob\left(\inf\{t\geq 0: B_t - (k-1)t =a-k\} < \infty
\right)\\
 &=e^{-2(k-1)(a-k)},
\end{align*}
for sufficiently large $a$. In the first step, we used the strong Markov property at the stopping time $U=\inf\{t \ge \vtau_{-g+k}: V_t=-g+k, S_t=X_t\}$ which satisfies $\vtau_{-g+k} \le U < \vtau_{-g+a}$ on the event $\{\vtau_{-g+k} < \vtau_{-g+a}\le \vtau_{-g+k-1}\}$,
by Remark \ref{oc16.1}. 
We also used Remark \ref{oc18.1}.
For the last step, we used 
\eqref{oc13.2}.
 For $a \ge 4$, $(k-1)(a-k) \ge a/2$ for $2 \le k \le a-1$ and thus, 
substituting the above estimate back into \eqref{eqnarray:probdec}, we get
$$\Prob(\vtau_{-g+a} < \vtau_{-g+1} \mid V_0=-g+a-1, H_0=0) \le (a-2)e^{-a}$$
for $a \ge 4$. This, along with \eqref{eqnarray:markov}, gives us
for  $v \in [-g+2,-g+a-1]$,
\begin{align*}
\Prob(\vtau_{-g+a} < \vtau_{-g+1} \mid V_0=v, H_0=0)
\le \Prob(\vtau_{-g+a-1} < \vtau_{-g+1} \mid V_0=v, H_0=0) (a-2)e^{-a}.
\end{align*}
For $a \ge 10$,  we apply the
above estimate inductively with $a$ replaced by $a/2, a/2+1,  \dots, a/2+k_*$,
where $k_*$ is the largest integer such that $a/2+k_* \leq a$,
to obtain
\begin{eqnarray*}
\sup_{v \in [-g+2,-g+a/2]}\Prob(\vtau_{-g+a} < \vtau_{-g+1} \mid V_0=v, H_0=0)&\le& C(a-2)^{a/2+1}
\exp\left(-\sum_{k=\lfloor a/2 \rfloor}^{\lfloor a \rfloor} k\right),
\end{eqnarray*}
which implies the  upper bound in \eqref{equation:upperfluc}.

\medskip

Finally, we proceed to prove the lower bound in \eqref{equation:upperfluc}. By the strong Markov property applied at $\vtau_v$ and Remarks \ref{oc18.1} and \ref{oc16.1}, for any $v \in [-g+2,-g+a/2]$,
$$
\Prob(\vtau_{-g+a} < \vtau_{-g+1} \mid V_0=-g+2, H_0=0)  \le  \Prob(\vtau_{-g+a} < \vtau_{-g+1} \mid V_0=v, H_0=0) .
$$
Thus, it suffices to work with the case $V_0=-g+2$. For $t < 1/g$, we have $V_t \ge V_0-gt > -g+1$. Hence $\vtau_{-g+1} \ge 1/g$. 

Suppose the event $\left\lbrace B_{1/g}>a\left(1+1/g\right)-2\right\rbrace$ holds. Then we claim that $\displaystyle{\sup_{t\le1/g}V_t > -g+a}$.
Suppose without loss of generality that $B_0=S_0=0$. If $\displaystyle{\sup_{t\le1/g}V_t \le -g+a}$, then $S_{1/g} \le (-g+a)/g$. Therefore, by \eqref{oc9.1},
$$L_{1/g}=\sup_{u\le 1/g}(B_u-S_u) \ge B_{1/g}-S_{1/g} >a\left(1+\frac{1}{g}\right)-2-\frac{(-g+a)}{g}=a-1.$$
Thus,
$$V_{1/g}=-g+2+L_{1/g}-g(1/g)>-g+2+(a-1)-1=-g+a.$$
which gives a contradiction. Hence, the event $\left\lbrace B_{1/g}>a\left(1+1/g\right)-2\right\rbrace$ implies the event $\{\vtau_{-g+a} < \vtau_{-g+1}\}$. Therefore, 
for  $v \in [-g+2,-g+a/2]$,
$$
\Prob(\vtau_{-g+a} < \vtau_{-g+1} \mid V_0=v, H_0=0)
 \ge \Prob(B_{1/g}>a(1+1/g)-2) \ge Ce^{-C'a^2},$$
which gives the lower bound in \eqref{equation:upperfluc}.
\end{proof}

The following theorem relates oscillations of the gap process $H_t$ to those of $V_t$.

\begin{thm}\label{thm:gaponeside}
There exist positive constants $\delta, a_0$ and $C^{*+}_i,C^{*-}_i$ for $i=1,2,3,4$ such that for any $a\ge a_0$:
\begin{equation}\label{equation:lowerflucgap}
C^{*-}_1e^{-C^{*-}_2a^{3/2}} \le \Prob(\stau_a < \vtau_{-g-1} \mid V_0=v, H_0=0) \le C^{*-}_3e^{-C^{*-}_4a^{3/2}}
\end{equation}
uniformly over all $v \in [-\delta \sqrt{a}, -g-2]$
and
\begin{equation}\label{equation:upperflucgap}
C^{*+}_1e^{-C^{*+}_2a} \le \Prob(\stau_a < \vtau_{-g+1} \mid V_0=v, H_0=0) \le C^{*+}_3e^{-C^{*+}_4a}
\end{equation}
uniformly over all $v \in [-g+2, \delta \sqrt{a}]$.
\end{thm}
\begin{proof}

Observe that
\begin{align*}
\sup_{v \in [-\sqrt{a}/4, -g-2]}
&\Prob(\stau_{2a/g} < \vtau_{-g-1}
 \mid V_0=v, H_0=0)\\
& \le \sup_{v \in [-\sqrt{a}/4, -g-2]}\Prob(\stau_{2a/g} < \vtau_{-g-1} \le \vtau_{-\sqrt{a}/2}\mid V_0=v, H_0=0)\\
& \quad + \sup_{v \in [-\sqrt{a}/4, -g-2]}\Prob( \vtau_{-\sqrt{a}/2} < \vtau_{-g-1}\mid V_0=v, H_0=0).
\end{align*}
By Theorem \ref{thm:fluconeside}, the second term is bounded above by $Ce^{-C'a^{3/2}}$. So, we estimate the first term. By the strong Markov property applied at time $\stau_{2a/g}$, we can write
\begin{align*}
\sup_{v \in [-\sqrt{a}/4, -g-2]}
&\Prob(\stau_{2a/g} < \vtau_{-g-1} \le \vtau_{-\sqrt{a}/2} \mid V_0=v, H_0=0)\\
& \le \sup_{v \in (-\sqrt{a}/2, -g-1)}\Prob(\vtau_{-g-1} \le \vtau_{-\sqrt{a}/2}\mid V_0=v, H_0=2a/g).
\end{align*}
Recall that $\sigma(u)=\inf\{t \ge u: S_t=X_t\}$.
We will argue that for any $v \in (-\sqrt{a}/2, -g-1)$,
\begin{align*}
\Prob(\vtau_{-g-1} \le \vtau_{-\sqrt{a}/2}\mid V_0=v, H_0=2a/g) & \le \Prob(\sigma(0) < \sqrt{a}/g \mid V_0=v, H_0=2a/g)\\
& \le \Prob(\sigma(0) < \sqrt{a}/g \mid V_0=-\sqrt{a}/2, H_0=2a/g).
\end{align*}
The first inequality above follows from the fact that if $\sigma(0) \ge \sqrt{a}/g$ then $V_t$ is strictly decreasing on $t \in [0,\sqrt{a}/g]$ and $V_{\sqrt{a}/g} < -g-1-\sqrt{a}$. As $V_0<-g-1$, this implies $\vtau_{-\sqrt{a}/2}< \vtau_{-g-1}$.

The second inequality holds because the first collision time between the inert particle and the Brownian particle with starting configuration $V_0=v'>-\sqrt{a}/2, H_0=2a/g$ stochastically dominates the first collision time starting from $V_0=-\sqrt{a}/2, H_0=2a/g$.

Note that if $V_0=-\sqrt{a}/2$, then for any $t \in [0,\sqrt{a}/g]$,
\begin{align}\label{oc29.1}
S_t \ge S_0- \frac{\sqrt{a}}{2}\cdot \frac{\sqrt{a}}{g}-\frac{g}{2}\left(\frac{\sqrt{a}}{g}\right)^2=S_0-\frac{a}{g}.
\end{align}
As $L_t$ is non-negative,
$$\sup_{t \le \sqrt{a}/g}X_t \le X_0+ \sup_{t \le \sqrt{a}/g}B_t.
$$
Thus, for $\sigma(0) < \sqrt{a}/g$ to hold, $X_t$ must hit the level $S_0-(a/g)$ for some $t \in [0, \sqrt{a}/g]$. Hence, $\sup_{t \le \sqrt{a}/g} X_t \ge S_0-(a/g)$ which, by the above, implies $\sup_{t \le \sqrt{a}/g} B_t \ge S_0-X_0-(a/g)$. Therefore, using \eqref{oc13.3},
\begin{align}\label{align:uppar}
\Prob(\sigma(0) < \sqrt{a}/g \mid V_0=-\sqrt{a}/2, H_0=2a/g) & \le \Prob\left(\sup_{t \le \sqrt{a}/g}B_t \ge \frac{a}{g}\right) \le Ce^{-C'a^{3/2}}
\end{align}
proving the upper bound in \eqref{equation:lowerflucgap}. 

\medskip

To prove the lower bound, we proceed similarly as the proof of the lower bound in \eqref{equation:lowerfluc}. Note that by the strong Markov property
applied at $\vtau_{-g-2}$ and Remarks \ref{oc18.1} and \ref{oc16.1}, for any $v \in [-\sqrt{a}/4, -g-2]$,
\begin{align*}
&\Prob(\stau_{2a/g} < \vtau_{-g-1} \mid V_0=v, H_0=0)\\
&\ge \Prob(\vtau_{-g-2} <\stau_{2a/g} \mid V_0=v, H_0=0)
\Prob(\stau_{2a/g}<\vtau_{-g-1}\mid V_0=-g-2, H_0=0).
\end{align*}
By the upper bound in \eqref{equation:lowerflucgap}, for sufficiently large $a$,
\begin{align*}
\inf_{v \in [-\sqrt{a}/4, -g-2]}&\Prob(\vtau_{-g-2} <\stau_{2a/g} 
\mid V_0=v, H_0=0)\\
& \ge \inf_{v \in [-\sqrt{a}/4, -g-2]}\Prob(\vtau_{-g-1} \le \stau_{2a/g} \mid V_0=v, H_0=0) \ge \frac{1}{2}.
\end{align*}
Thus, it suffices to assume that the starting configuration is $V_0=-g-2, H_0=0$. Fix $\eps \in (0,1/g)$ such that $\displaystyle{(g+2)\eps+\frac{g\eps^2}{2}<\frac{1}{2}}$ and define the event $A$ as in \eqref{equation:eventdef}. Recall that $\Prob(A) =q>0$
and $A \subset\{\vtau_{-g-1}>\eps, S_{\eps}-X_{\eps}>\eps\}$ (see \eqref{oc27.1}).  Note that $S_t - S_{\eps} \ge -(g+3)t - \frac{1}{2}gt^2$ for $t \ge \eps$.

Define the downward parabola
$$
y(t)=\eps -\frac{\sqrt{a}}{2}t-\frac{1}{2}gt^2.
$$
Our strategy will be to show that with probability bounded below by $Ce^{-C'a^{3/2}}$, the process $X_t-X_{\eps}$ remains below the downward parabola $y(t-\eps)$
for all $t \in [\eps, \sqrt{a}/g + \eps]$. Assuming that this holds and using the fact that the gap between $S_t$ and $y(t-\eps)$ increases with $t$, we will conclude that the gap between $S_t$ and $X_t$ increases as well. Since $y(t-\eps)$ lies strictly below $S(t)$ on $ (\eps, \sqrt{a}/g + \eps]$, there is no collision between $S$ and $X$ and the velocity $V$ decreases strictly, ensuring that the velocity stays below the level $-g-1$ on this time interval.

Define the event
$$
F=\{X_t-X_{\eps} \le y(t-\eps) \text{ for all } t \in [\eps, \sqrt{a}/g + \eps]\}.
$$
If $A \cap F$ holds, 
then $\vtau_{-g-1}>\sqrt{a}/g + \eps$. Since $\eps<1/g$, $V_{\eps} \ge -(g+3)$ and 
$$
S_{\sqrt{a}/g + \eps} \ge S_{\eps} -(g+3)\frac{\sqrt{a}}{g}-\frac{1}{2}g\left(\frac{\sqrt{a}}{g}\right)^2 \ge S_{\eps}-\frac{3a}{4g},
$$
for sufficiently large $a$. Recalling that $H_{\eps}=S_{\eps}-X_{\eps}>\eps$, we obtain
\begin{align*}
S_{\sqrt{a}/g + \eps}-X_{\sqrt{a}/g + \eps} &\ge S_{\eps}-\frac{3a}{4g}-X_{\eps}-y(\sqrt{a}/g)\\
&= (S_{\eps}-X_{\eps})-\frac{3a}{4g}-\eps +\frac{\sqrt{a}}{2}\cdot \frac{\sqrt{a}}{g} +\frac{1}{2}g\left(\frac{\sqrt{a}}{g}\right)^2 > \frac{a}{4g}.
\end{align*}
Therefore, under $A \cap F$, it holds that $\stau_{a/4g} < \sqrt{a}/g + \eps < \vtau_{-g-1}$.
Thus, to prove the lower bound in \eqref{equation:lowerflucgap}, it will suffice to show that $\Prob(A \cap F) \ge Ce^{-C'a^{3/2}}$.

By the Markov property applied at time $\eps$,
\begin{align}\label{align:intersect}
\Prob(A \cap F) &\ge q\inf_{v' \ge -(g+3), z > \eps} \Prob(X_t \le y(t) \text{ for all } t \in [0, \sqrt{a}/g] \mid V_0= v', H_0=z)\nonumber\\
&\ge q\Prob(X_t \le y(t) \text{ for all } t \in [0, \sqrt{a}/g] \mid V_0= -g-3, H_0=\eps),
\end{align}
where the last inequality follows by a stochastic domination argument similar to that in \eqref{eqnarray:lowerbd}.

By  \cite[Thm.~2.4]{hofstad} (see especially (2.4)), there is $C>0$  such that
\begin{equation}\label{equation:hofstad}
\Prob(X_t \le y(t) \text{ for all } t \in [0, \sqrt{a}/g] \mid V_0= -g-3, H_0=\eps) \ge C\int_{\sqrt{a}/g}^{\infty}t^{1/2}
e^{-Z(t)}dt,
\end{equation}
where
\begin{align*}
Z(t) = \left( \rpar{t+\frac{\sqrt{a}}{2g}}^3 - \frac{t^3}{4}-\rpar{\frac{\sqrt{a}}{2g}}^3 - 6\frac{\sqrt{a}}{2g} \frac{\eps}{g} - 3\frac{\eps}{g} t
\right) \frac{g^2}{6}.
\end{align*}
For $t \ge \sqrt{a}/g$, we have $\sqrt{a}/(2g) \le t/2$. A routine calculation then shows that $Z(t) \le C t^3$. Substituting this in \eqref{equation:hofstad}, we get
$$
\Prob(X_t \le y(t) \text{ for all } t \in [0, \sqrt{a}/g] \mid V_0= -g-3, H_0=\eps) \ge Ce^{-C'a^{3/2}}.
$$
This estimate and \eqref{align:intersect} imply that $\Prob(A \cap F) \ge Ce^{-C'a^{3/2}}$ which gives the lower bound in \eqref{equation:lowerflucgap}.

\medskip

Next we will prove the upper bound in \eqref{equation:upperflucgap}.
We apply Lemma \ref{lem:unifcontrol} with $N=3$, replacing $a$ in the lemma by $\sqrt{a}$ and taking $\delta$ there to be $1/(3\sqrt{g})$. The lemma implies that we can find positive constants $C, C', a_0$ such that
\begin{equation}\label{equation:delsqrt}
\Prob\left(\sigma(t) > t +\frac{\sqrt{a}}{g} \text{ for some } t \le a^{3/2} \wedge \vtau_{-g+\frac{\sqrt{a}}{12} }\wedge \vtau_{-g-\frac{a}{72g}} \mid V_0=v, H_0=0\right) \le Ce^{-C'a^{3/2}},
\end{equation}
for all $a \ge a_0$ and all $v \in [-g-\frac{a}{72g},-g+\frac{\sqrt{a}}{12}]$. For any $0<\delta<\frac{1}{48}$, \eqref{equation:delsqrt} implies
\begin{equation}\label{equation:bleh}
\sup_{v \in [-g+2, \delta \sqrt{a}]}\Prob\left(\sigma(\stau_{2a/g}) > \stau_{2a/g} +\frac{\sqrt{a}}{g}, \ \stau_{2a/g} \le a^{3/2} \wedge \vtau_{4\delta\sqrt{a}}\wedge \vtau_{-g+1} \mid V_0=v, H_0=0\right) \le Ce^{-C'a^{3/2}},
\end{equation}
for sufficiently large $a$. Next we apply Lemma \ref{lem:hitup} with $a_1=1$ and $a_2=g+v$. We note that the bounds in that lemma do not depend on $a_2$ provided $t \ge 2(a_2-a_1)/(g \wedge a_1)$. We obtain
\begin{equation}\label{equation:hithere}
\sup_{v \in [-g+2, \delta \sqrt{a}]}
\Prob\left(\vtau_{-g+1} > a^{3/2} \mid V_0=v, H_0=0\right) \le Ce^{-C'a^{3/2}}.
\end{equation}
From \eqref{equation:bleh} and \eqref{equation:hithere}, we see that
\begin{align}\label{equation:unif}
&\sup_{v \in [-g+2, \delta \sqrt{a}]} \Prob\left(\sigma(\stau_{2a/g}) > \stau_{2a/g} + \frac{\sqrt{a}}{g}, \stau_{2a/g} \le \vtau_{-g+1} \leq \vtau_{4\delta \sqrt{a}}\mid V_0=v, H_0=0\right)\nonumber \\
&\leq \sup_{v \in [-g+2, \delta \sqrt{a}]} \Prob\left(\sigma(\stau_{2a/g}) > \stau_{2a/g} + \frac{\sqrt{a}}{g}, \stau_{2a/g} \le \vtau_{-g+1} \wedge \vtau_{4\delta \sqrt{a}}\mid V_0=v, H_0=0\right)\nonumber\\
& \le C e^{-C'a^{3/2}},
\end{align}
for sufficiently large $a$.
Applying the strong Markov property at $\stau_{2a/g}$,
\begin{align}\label{align:eq2}
\sup_{v \in [-g+2, \delta \sqrt{a}]}
&\Prob\left(\stau_{2a/g}  < \vtau_{-g+1} \le \vtau_{4\delta \sqrt{a}}, \ \sigma(\stau_{2a/g}) \le \stau_{2a/g} + \sqrt{a}/g \mid V_0=v, H_0=0\right)\nonumber\\
& \le \sup_{v' \in (-g+1, 4\delta \sqrt{a})}
\Prob\left(\sigma(0) \le \sqrt{a}/g \mid V_0=v', H_0=2a/g\right)\nonumber\\
& \le \Prob\left(\sigma(0) \le \sqrt{a}/g \mid V_0=-g+1, H_0=2a/g\right).
\end{align}

Note that if $V_0=-g+1$, then for sufficiently large $a$ and all $t \in [0,\sqrt{a}/g]$,
$$
S_t \ge S_0-\left(-g+1\right)\frac{\sqrt{a}}{g}-\frac{g}{2}\left(\frac{\sqrt{a}}{g}\right)^2=S_0-\frac{a}{g}.
$$
The argument following \eqref{oc29.1} can now be repeated verbatim to show the following analogue of \eqref{align:uppar},
\begin{align*}
\Prob\left(\sigma(0) \le \sqrt{a}/g \mid V_0=-g+1, H_0=2a/g\right) \le Ce^{-C'a^{3/2}}.
\end{align*}
This, \eqref{equation:unif} and \eqref{align:eq2} yield
\begin{equation}\label{equation:eqC}
\sup_{v \in [-g+2, \delta \sqrt{a}]}
\Prob\left(\stau_{2a/g} < \vtau_{-g+1} \le \vtau_{4\delta \sqrt{a}} \mid V_0=v, H_0=0\right) \le Ce^{-C'a^{3/2}}.
\end{equation}

We have
\begin{align}\label{align:eqA}
\sup_{v \in [-g+2, \delta \sqrt{a}]}&
\Prob\left(\stau_{2a/g} < \vtau_{-g+1}  \mid V_0=v, H_0=0\right)\nonumber\\
& \le \sup_{v \in [-g+2, \delta \sqrt{a}]}
\Prob\left(\stau_{2a/g} < \vtau_{-g+1} \le \vtau_{4\delta \sqrt{a}} \mid V_0=v, H_0=0\right)\nonumber\\
&\quad + \sup_{v \in [-g+2, \delta \sqrt{a}]}
\Prob\left(\vtau_{4\delta \sqrt{a}} < \vtau_{-g+1} \mid V_0=v, H_0=0\right).
\end{align}
We apply Theorem \ref{thm:fluconeside}, replacing $a$ by $4\delta\sqrt{a}$ and noting that for sufficiently large $a$, $-g+2\delta\sqrt{a}>\delta\sqrt{a}$. The theorem yields
\begin{equation}\label{equation:eqB}
\sup_{v \in [-g+2, \delta \sqrt{a}]}
\Prob\left(\vtau_{4\delta \sqrt{a}} < \vtau_{-g+1} \mid V_0=v, H_0=0\right) 
\le Ce^{-C'a}.
\end{equation}
Combining \eqref{equation:eqB} and \eqref{equation:eqC} with \eqref{align:eqA}, we get
\begin{equation*}
\sup_{v \in [-g+2, \delta \sqrt{a}]}\Prob(\stau_{2a/g} < \vtau_{-g+1} \mid V_0=v, H_0=0) \le Ce^{-C'a}
\end{equation*}
which gives the upper bound in \eqref{equation:upperflucgap}.

\medskip

Finally, to prove the lower bound in \eqref{equation:upperflucgap}, note that for any $v \in [-g+2, \sqrt{a}]$,
\begin{align*}
\Prob&\left(\stau_{a/g} < \vtau_{-g+1}  \mid V_0=v, H_0=0\right)\\
& \ge \Prob\left(\vtau_{3\sqrt{a}} < \vtau_{-g+1}, \ \sup_{t \in \left[\vtau_{3\sqrt{a}}, \ \vtau_{-g+1}\right)}H_t\ge a/g  \mid V_0=v, H_0=0\right).
\end{align*}
By the strong Markov property applied at time $\vtau_{3\sqrt{a}}$ to the right hand side above and Remarks \ref{oc18.1} and \ref{oc16.1}, we get
\begin{align}\label{align:lb1}
\Prob&\left(\stau_{a/g} < \vtau_{-g+1} \mid V_0=v, H_0=0\right)\\
 & \ge \Prob\left(\vtau_{3\sqrt{a}}< \vtau_{-g+1} \mid V_0=v, H_0=0\right) 
\Prob\left(\stau_{a/g} < \vtau_{-g+1} \mid V_0=3\sqrt{a}, H_0=0\right). \nonumber
\end{align}
For $V_0=3\sqrt{a}$ and $H_0=0$, $S_{\sqrt{a}/g}
\geq 3\sqrt{a}(\sqrt{a}/g) - (g/2) (\sqrt{a}/g)^2  > 2a/g$. For $t \in [0,\sqrt{a}/g]$,
$$
V_t \ge V_0-gt \ge 3\sqrt{a}-\sqrt{a}=2\sqrt{a}.
$$
Therefore, taking $a$ sufficiently large so that $2\sqrt{a}>-g+1$, it follows that $\frac{\sqrt{a}}{g} < \vtau_{-g+1}$. Thus,
\begin{align*}
\Prob\left(\stau_{a/g} < \vtau_{-g+1} \mid V_0=3\sqrt{a}, H_0=0\right) \ge \Prob\left(B_{\sqrt{a}/g} < a/g\right) \ge \frac{1}{2}
\end{align*}
for sufficiently large $a$. This and \eqref{align:lb1} give us
\begin{align*}
\inf_{v \in [-g+2, \sqrt{a}]}
\Prob\left(\stau_{a/g} < \vtau_{-g+1} \mid V_0=v, H_0=0\right) 
& \ge \frac{1}{2}\inf_{v \in [-g+2, \sqrt{a}]}
\Prob\left(\vtau_{3\sqrt{a}}< \vtau_{-g+1} \mid V_0=v, H_0=0\right).
\end{align*}
We apply Theorem \ref{thm:fluconeside} replacing $a$ with $4\sqrt{a}$ and noting that for sufficiently large $a$, $-g+4\sqrt{a}>3\sqrt{a}$ and $-g+2\sqrt{a}>\sqrt{a}$. We obtain
\begin{equation*}
\inf_{v \in [-g+2, \sqrt{a}]}\Prob\left(\vtau_{3\sqrt{a}}< \vtau_{-g+1} \mid V_0=v, H_0=0\right) \ge Ce^{-C'a},
\end{equation*}
which, in view of the previous estimate, completes the proof of the lower bound in \eqref{equation:upperflucgap} and, hence, the proof of the theorem.
\end{proof}

\begin{remark}\label{oc31.2}
We will sketch an argument showing that
for any initial values $V_0=v$, $X_0=x$ and $S_0=y$, and any $z\in \Reals$ we have $\vtau_z < \infty$, a.s. 

In view of Remark \ref{oc31.1}, we can assume that $H_0=0$.
By Lemma \ref{lem:tube}, the process $V$ cannot stay in any bounded interval forever, a.s.
By Lemmas \ref{lem:hittimeest} and \ref{lem:hitup}, the probability that $V$ converges to $\infty$ or $-\infty$ is 0, a.s.

Suppose that  $\limsup _{t\to \infty} V_t = \infty$ and 
$\liminf_{t\to\infty} V_t$ is finite with positive probability. 
Let $a\in \Reals$ and $\eps>0$ be such that $\Prob(\liminf_{t\to\infty} V_t\in (a, a+1))>\eps $. The methods used in our proofs show easily that for some $p>0$, all $v\in(a,a+1)$, and all $x<y$, if $V_0 = v$, $S_0 = y$ and $X_0=x$ then $\vtau _{a-3} < \infty$ with probability greater than $p$. It is now standard to prove that on the event where $V$ visits $(a,a+1)$ infinitely often, it has to visit $(a-2, a-1)$ infinitely often as well, and, therefore, 
$\liminf_{t\to\infty} V_t$ cannot lie in $(a,a+1)$ with positive probability, a contradiction. A similar argument shows that the event that 
$\liminf _{t\to \infty} V_t = -\infty$ and 
$\limsup_{t\to\infty} V_t$ is finite has probability 0.
\end{remark}

\section{Renewal times}\label{sec:renewal}
We will define several sequences of stopping times {and derive tail estimates for them} that will help us estimate the fluctuations of the velocity process $V_t$ and the gap process $S_t-X_t$. {The path $\{Z_s: s \le t\}$ will be decomposed into \textit{cycles} between consecutive renewal times defined in this section. In Section \ref{sec:strong}, fluctuation results will be established for these cycles. These, in turn, will yield global fluctuation results and strong laws for $S_t$ and $X_t$ stated in Theorem \ref{thm:velfluc} and Theorem \ref{thm:stronglaw}.}

We assume that the starting configuration is $V_0=-g$ and $H_0=0$, although the results that follow will not depend on this choice. Fix $a_0>g$. We  define a sequence of \textit{renewal times} $\{\zeta_k\}_{k \ge 0}$ as follows. Let $\zeta_0=0$ and for $k \ge 0$,
\begin{align}\label{align:eta}
\eta_k &=\inf\{t \ge \zeta_k: |V_t+g|=a_0+2\},\\
\zeta_{k+1}&=\inf\{t \ge \eta_k: V_t=-g \text{ and } S_t=X_t\}.
\label{align:zeta}
\end{align}
In Lemma \ref{lem:estzeta}, we will 
prove that all these stopping times are finite, a.s. Moreover, we will show that the distribution of $\zeta_1$ has a rapidly decaying tail and thus has finite moments of all orders.

Define $\alpha_{-1}=0$. If $\vtau_{-g+a_0+2}< \vtau_{-g-a_0-2}$, define $\alpha_k=0$ for all $k \ge 0$ and let $N^-=0$. On the event $\{\vtau_{-g-a_0-2}< \vtau_{-g+a_0+2}\}$, define $\alpha_0=\vtau_{-g-a_0-2}$. For $k \ge 0$, if $V_{\alpha_{3k}}=-g-a_0-2$, then define
\begin{align}\label{align:alphastop}
\alpha_{3k+1}&=\inf\{t \ge \alpha_{3k}: S_t=X_t\},\nonumber\\
\alpha_{3k+2}&=\inf\{t \ge \alpha_{3k+1}: V_t=-g-a_0-1\},\nonumber\\
\alpha_{3k+3}&=\inf\{t \ge \alpha_{3k+2}: V_t=-g \text{ or } -g-a_0-2\}.
\end{align}
If $V_{\alpha_{3k}}=-g$, then define $\alpha_j=\alpha_{3k}$ for all $j \ge 3k$. . Define $N^-=\inf\{ k \ge 1: V_{\alpha_{3k}}=-g\}$. 
This corresponds to the first hitting of $-g$ by the velocity after time $\vtau_{-g-a_0-2}$.
By Remark \ref{oc16.1}, if $V_{\alpha_{3k}}=-g$ then $S_{\alpha_{3k}}= X_{\alpha_{3k}}$. 
Thus, on the event $\{\vtau_{-g-a_0-2}< \vtau_{-g+a_0+2}\}$, we have that $\zeta_1=\alpha_{3N^-}$.
Also, note that $V_{\alpha_{3k+1}} \leq -g-a_0-2$, and $H_{\alpha_{3k+2}}=0$ for $k<N^-$.

Define $\beta_{-1}=0$. If $\vtau_{-g-a_0-2}< \vtau_{-g+a_0+2}$, define $\beta_k=0$ for all $k \ge 0$ and let $N^+=0$. On the event $\{\vtau_{-g+a_0+2}< \vtau_{-g-a_0-2}\}$, define $\beta_0=\vtau_{-g+a_0+2}$. For $k \ge 0$, if $V_{\beta_{3k}}=-g+a_0+2$, then define
\begin{align*}
\beta_{3k+1}&=\inf\{t \ge \beta_{3k}: V_t=-g+a_0+1\},\nonumber\\
\beta_{3k+2}&=\inf\{t \ge \beta_{3k+1}: S_t=X_t \text{ or } V_t=-g\},\nonumber\\
\beta_{3k+3}&=\inf\{t \ge \beta_{3k+2}: V_t=-g \text{ or } -g+a_0+2\}.
\end{align*}
Otherwise, if $V_{\beta_{3k}}=-g$, define $\beta_j=\beta_{3k}$ for all $j \ge 3k$.
Define $N^+=\inf\{ k \ge 1: V_{\beta_{3k}}=-g\}$. This corresponds to the first down-crossing of the velocity below the level $-g$ after $\vtau_{-g+a_0+2}$. But with positive probability, $S_{\beta_{3N^+}}> X_{\beta_{3N^+}}$. Thus, to reach the renewal time $\zeta_1$, we will define a further set of stopping times $\{\talpha_k\}_{k \ge -1}$ till the first time the velocity hits $-g$ again from below and thus the processes $S$ and $X$ coincide.

If $\vtau_{-g-a_0-2}< \vtau_{-g+a_0+2}$, we define $\talpha_k=0$ for all $k \ge -1$ and we let $\tN^-=0$. On the event $\left\{\vtau_{-g+a_0+2}< \vtau_{-g-a_0-2}\right\}$, let
\begin{align*}
\talpha_{-1}&=\inf\{t \ge \beta_{3N^+}: S_t=X_t \text{ or } V_t=-g-a_0-2 \},\\
\talpha_{0}&=\inf\{t \ge \talpha_{-1}: V_t=-g \text{ or } -g-a_0-2\}.
\end{align*}
For $k \ge 0$, if $V_{\talpha_{3k}} = -g-a_0-2$, define $\talpha_i$ for $i=3k+1, 3k+2, 3k+3$ exactly as in \eqref{align:alphastop} replacing the $\alpha$'s with $\talpha$'s. If $V_{\talpha_{3k}} = -g$, define $\talpha_j=\talpha_{3k}$ for all $j \ge 3k$. Let $\tN^-=\inf\{ k \ge 0: V_{\talpha_{3k}}=-g\}$. Thus, on the event $\{\vtau_{-g+a_0+2}< \vtau_{-g-a_0-2}\}$, we have $\zeta_1=\talpha_{3\tN^-}$.

We will argue that all stopping times $\alpha_k$ are finite a.s.
First of all, by Lemma \ref{lem:tube},
\begin{align*}
\Prob\left(
\{\vtau_{-g-a_0-2}< \infty\}
\cup
\{\vtau_{-g+a_0+2} < \infty\}
\right)=1,
\end{align*}
so at least one of the sequences $\{\alpha_k\}$ or $\{\beta_k\}$ is non-trivial.
Suppose that $\alpha_{3k} < \infty$, a.s. Then
$\alpha_{3k+1} < \infty$, a.s., by Remark \ref{oc31.1}.
If $\alpha_{3k+1} < \infty$, a.s. then
$\alpha_{3k+2} < \infty$, a.s., by Remark \ref{oc31.2}.
Since  the argument in that remark was only sketched, note that 
we can alternatively apply Lemma \ref{lem:hittimeest} which has a 
detailed proof. Finally, if $\alpha_{3k+2} < \infty$, a.s. then
$\alpha_{3k+3} < \infty$, a.s., by Lemma \ref{lem:tube}.
A similar argument applies to $\beta_k$'s and $\talpha_k$'s.

\begin{lem}\label{lem:estzeta}
Suppose that $V_0=-g$ and $H_0=0$.
There exist constants $C, C'>0$ such that for all $t\geq 0$,
\begin{equation*}
\Prob(\zeta_1>t) \le Ce^{-C't^{1/2}}.
\end{equation*}
It follows that for any integer $n \ge 1$, $\Expect(\zeta_1^n) < \infty$.
\end{lem}
\begin{proof}
In this proof, we  will  assume that $t$ is sufficiently large without explicitly mentioning it every time.
 
First, we consider the event $\left\{\vtau_{-g-a_0-2}< \vtau_{-g+a_0+2}\right\}$. Write 
\begin{align*}
p^-=\Prob\left(\vtau_{-g-a_0-2}<\vtau_{-g} \mid V_0=-g-a_0-1, H_0=0\right).
\end{align*}

Note that if $V_0=-g-a_0-1, H_0=0$,
then for $t \le \vtau_{-g}$, we have $L_t \ge L^{(-g)}_t$. 
Note that
\begin{align*}
\Prob\left(\inf\{t\geq 0: L^{(-g)}_t-gt = -1\} \leq 1/g\right)
\leq
\Prob\left( \sup_{t \le 1/g}(B_t+gt) - g(1/g) < -1 \right) = 0.
\end{align*}
This implies that
\begin{align*}
p^- & \le \Prob\left(\inf\{t\geq 0: L^{(-g)}_t-gt = -1\} < 
\inf\{t\geq 0: L^{(-g)}_t-gt = a_0+1\}\right)\\
& \le \Prob\left( \sup_{t \le 1/g}(B_t+gt) - g (1/g) \le a_0+1\right)<1.
\end{align*}
Hence, for any integer $n \ge 1$, we can apply the strong Markov property successively at $\alpha_{3n-1}, \alpha_{3n-4}, \dots$ to get
\begin{equation}\label{equation:estN}
\Prob(N^->n) \le (p^-)^n.
\end{equation}
For $t>0$ and $n \ge 1$,
\begin{align} \label{align:nless}
\Prob\left(1 \le N^- \le n, \sup_{0 \le k \le N^- -1}(\alpha_{3(k+1)}-\alpha_{3k}) > 3t\right) &\le \sum_{k=0}^{n-1}\Prob\left(\alpha_{3(k+1)}-\alpha_{3k} > 3t, \ N^- \ge k+1\right).
\end{align}
For all $k $ and $t \in [\alpha_{3k-1}, \alpha_{3k}]$, we have $V_t \in [-g-a_0-2,-g+a_0+2]$.
Therefore, by the strong Markov property applied at $\alpha_{3k-1}$, and Lemma \ref{lem:tube}, there exist constants $C>0$ and $p_0 \in (0,1)$ such that for any integer $m \ge 1$,
\begin{align}\label{align:b1}
&\Prob(\alpha_{3k}  -\alpha_{3k-1} >Cm, \ N^- \ge k+1)
\leq
\Prob(\alpha_{3k}  -\alpha_{3k-1} >Cm)
\nonumber\\
& \le \Prob(|V_s +g| \le a_0+2 \text{ for all } s \in [0, Cm] \mid V_0=-g-a_0-1, H_0=0) \le p_0^m.
\end{align}

If $V_{\alpha_{3k-3}} = -g$ then $\alpha_{3k+1}-\alpha_{3k}=0$. 
If $V_{\alpha_{3k-3}} = -g-a_0-2$ then
$V_{\alpha_{3k-1}} = -g-a_0-1$ and, by Remark \ref{oc16.1}, $S_{\alpha_{3k-1}} = X_{\alpha_{3k-1}}$.
These remarks and the strong Markov property applied at time $\alpha_{3k-1}$ 
show that for any $\delta>0$ and sufficiently large $t$,
\begin{align}\label{align:b2}
&\Prob(\alpha_{3k+1}-\alpha_{3k}>\delta t^{1/3}, \alpha_{3k}-\alpha_{3k-1} \le t, \ N^- \ge k+1)\nonumber\\
&\leq
\Prob(\alpha_{3k+1}-\alpha_{3k}>\delta t^{1/3}, \alpha_{3k}-\alpha_{3k-1} \le t)
\nonumber\\
&= \Prob(\sigma(\alpha_{3k})-\alpha_{3k} > \delta t^{1/3}, \alpha_{3k}-\alpha_{3k-1} \le t)\nonumber\\
& \le  \Prob\left(\sigma(u) > u +\delta t^{1/3} \text{ for some } u \le t \wedge \vtau_{-g-a_0-2} \wedge \vtau_{-g+a_0+2} \mid V_0= -g-a_0-1, H_0=0\right)\nonumber\\
& \le Ce^{-C't}.
\end{align}
The last estimate follows from Lemma \ref{lem:unifcontrol} by applying it with $a =\sqrt{g}t^{1/3}/3$  and $m=3$.
We combine \eqref{align:b1}, taking $m=t/C$ there, and \eqref{align:b2}, to obtain, 
\begin{equation}\label{equation:estdip}
\Prob(\alpha_{3k+1}-\alpha_{3k}>\delta t^{1/3}, \ N^- \ge k+1) \le Ce^{-C't}.
\end{equation}
We claim that for sufficiently large $t$,
\begin{align}\label{align:estup}
\Prob(\alpha_{3k+2}&-\alpha_{3k+1}>t, \alpha_{3k+1}-\alpha_{3k} \le \delta t^{1/3}, \ N^- \ge k+1)\nonumber\\
&\le \sup_{v \in [a_0+2, a_0+2+g\delta t^{1/3}]} \Prob(\vtau_{-g-a_0-1} >t \mid V_0=-g-v, H_0=0)\nonumber\\
&\le \sup_{v \in [a_0+2, a_0+2+g\delta t^{1/3}]}\frac{4(v-a_0-1)}{((v-a_0-1) +(a_0+1)t)(a_0+1)\sqrt{2\pi t}}e^{-(a_0+1)^2t/8} \le Ce^{-C't}.
\end{align}
The first inequality follows from the strong Markov property applied at $\alpha_{3k+1}$. For the second inequality, we apply Lemma \ref{lem:hittimeest} with $a_1=a_0+1$ and $a_2=v$. Note that as $\frac{2(v-a_0-1)}{(a_0+1)} \le \frac{2(1+g\delta t^{1/3})}{(a_0+1)} < t$ for sufficiently large $t$, the hypotheses of Lemma \ref{lem:hittimeest} are satisfied.

By Lemma \ref{lem:tube},
\begin{equation}\label{equation:esttube}
\Prob(\alpha_{3k+3}-\alpha_{3k+2}>t, \ N^- \ge k+1) \le 
\Prob(\alpha_{3k+3}-\alpha_{3k+2}>t) \le Ce^{-C't}.
\end{equation}
From \eqref{equation:estdip}, \eqref{align:estup} and \eqref{equation:esttube}, we get
\begin{equation*}
\Prob\left(\alpha_{3(k+1)}-\alpha_{3k} > 3t, \ N^- \ge k+1\right) \le Ce^{-C't}.
\end{equation*}
Substituting this into \eqref{align:nless}, we obtain
\begin{equation}\label{equation:n1}
\Prob\left(1 \le N^- \le t, \sup_{0 \le k \le N^- -1}(\alpha_{3(k+1)}-\alpha_{3k}) > 3t\right) \le Cte^{-C't}.
\end{equation}

Recall that if $\vtau_{-g-a_0-2}<\vtau_{-g+a_0+2}$ then $\zeta_1=\alpha_{3N^-}$. Thus, using \eqref{equation:estN} and \eqref{equation:n1},
\begin{align*}
\Prob\left(\zeta_1>3t^2, \vtau_{-g-a_0-2}<\vtau_{-g+a_0+2}\right) &= \Prob(\alpha_{3N^-}>3t^2, \ N^- \ge 1)\\
& \le \Prob(N^- > t) + \Prob\left(\sum_{k=0}^{N^- -1}(\alpha_{3(k+1)}-\alpha_{3k})>3t^2, 1 \le N^- \le t\right)\\
& \le \Prob(N^- > t) + \Prob\left(1 \le N^- \le t, \sup_{0 \le k \le N^- -1}(\alpha_{3(k+1)}-\alpha_{3k}) > 3t\right)\\
& \le (p^-)^t + Cte^{-C't}.
\end{align*}
After readjustment of constants we get
\begin{equation}\label{equation:thetaminus}
\Prob\left(\zeta_1>t^2, \vtau_{-g-a_0-2}<\vtau_{-g+a_0+2}\right) \le Ce^{-C't}.
\end{equation}
We record a related estimate for later use. Note that  
$\inf\{t\geq \alpha_2: V_t = -g\} \leq \zeta_1$.
Hence our argument proving \eqref{equation:thetaminus} also shows
\begin{equation}\label{equation:hitg}
\Prob(\vtau_{-g} >t^2 \mid V_0=-g-a_0-1, H_0=0) \le Ce^{-C't}.
\end{equation}

\medskip

Next consider the event $\left\{\vtau_{-g+a_0+2}< \vtau_{-g-a_0-2}\right\}$. Note that if we start with $V_0=v \in [-g, -g+a_0+1]$ and $H_0=0$, then for $t \le \vtau_{-g}$, it holds that $L_t \le L^{(-g)}_t$. Therefore, for some $p^+ < 1$,
\begin{align*}
\inf_{v \in [-g, -g+a_0+1]}
&\Prob(\vtau_{-g} < \vtau_{-g+a_0+2} \mid V_0=v, H_0=0)\\
&\geq \Prob\left(\inf\{t\geq 0: L^{(-g)}_t-gt = 1\} > 
\inf\{t\geq 0: L^{(-g)}_t-gt = -a_0-1\}\right)\\
& \ge \Prob\left(\sup_{t \le 2(a_0+1)/g}(B_t+gt)<1\right)
 =1-p^+  >0.
\end{align*}
Therefore, for any integer $n \ge 1$, we can apply the strong Markov property successively at times $\beta_{3n-1}, \beta_{3n-4}, \dots$ to get
\begin{equation}\label{equation:estNplus}
\Prob(N^+>n) \le \left(\sup_{v \in [-g, -g+a_0+1]}\Prob(\vtau_{-g} > \vtau_{-g+a_0+2} \mid V_0=v, H_0=0)\right)^n \le (p^+)^n.
\end{equation}
As before, we can write
\begin{align} \label{align:nplusless}
\Prob\left(1 \le N^+ \le n, \sup_{0 \le k \le N^+ -1}(\beta_{3(k+1)}-\beta_{3k}) > 3t\right) &\le \sum_{k=0}^{n-1}\Prob\left(\beta_{3(k+1)}-\beta_{3k} > 3t, \ N^+ \ge k+1\right).
\end{align}
Let $0 \le k \le n-1$. By Lemma \ref{lem:hitup}, for large $t$, 
\begin{equation}\label{equation:upone}
\Prob(\beta_{3k+1}-\beta_{3k} > t, \ N^+\ge k+1) \le \Prob(\beta_{3k+1}-\beta_{3k} > t) \le Ce^{-C't}.
\end{equation}
For $N^+ \ge k+1$, we have
$$
-g<V_{\beta_{3k+2}}=-g+a_0+1-g(\beta_{3k+2}-\beta_{3k+1})
$$
yielding
\begin{equation}\label{equation:uptwo}
\beta_{3k+2}-\beta_{3k+1}<\frac{a_0+1}{g}.
\end{equation}
By Lemma \ref{lem:tube},
\begin{equation}\label{equation:upthree}
\Prob(\beta_{3k+3}-\beta_{3k+2}>t, \ N^+ \ge k+1) \le
\Prob(\beta_{3k+3}-\beta_{3k+2}>t) \le Ce^{-C't}.
\end{equation}
Combining \eqref{align:nplusless}, \eqref{equation:upone}, \eqref{equation:uptwo} and \eqref{equation:upthree}, we get
\begin{equation*}
\Prob\left(1 \le N^+ \le t, \sup_{0 \le k \le N^+ -1}(\beta_{3(k+1)}-\beta_{3k}) > 3t\right) \le Ct e^{-C't}.
\end{equation*}
This and \eqref{equation:estNplus}
can be combined as in the proof of
 \eqref{equation:thetaminus} to show that for large $t$, 
\begin{equation}\label{equation:thetaplusup}
\Prob(\beta_{3N^+}>t^2, \vtau_{-g+a_0+2}<\vtau_{-g-a_0-2}) \le Ce^{-C't}.
\end{equation}
The following estimate, needed later in the paper, can be derived
just like the last estimate:
\begin{equation}\label{no8.1}
\Prob(\vtau_{-g} >t^2 \mid V_0=-g+a_0+1, H_0=0) \le Ce^{-C't}.
\end{equation}

\medskip

Next we will estimate the remaining time $\zeta_1-\beta_{3N^+}$ before renewal happens.
First consider the event $\{N^+ \ge 1, \ \tN^-=0\}$ where the first renewal time $\zeta_1$ is reached before the velocity hits level $-g-a_0-2$. Under this event, there are the following two possibilities. If $S_{\beta_{3N^+}}=X_{\beta_{3N^+}}$, then $\zeta_1=\beta_{3N^+}$. Otherwise, $V_{\talpha_{-1}} \in (-g-a_0-2,-g)$, $V_{\talpha_0}=-g$ and $\zeta_1=\talpha_0$. As the inert particle falls freely in the time interval $[\beta_{3N^+}, \talpha_{-1}]$, we have
\begin{equation}\label{equation:e1}
\talpha_{-1}-\beta_{3N^+} \le (a_0+2)/g.
\end{equation}
By the strong Markov property applied at $\talpha_{-1}$ and Lemma \ref{lem:tube},
\begin{equation}\label{equation:e2}
\Prob(\talpha_0-\talpha_{-1}>t, \ N^+ \ge 1, \ \tN^-=0) \le Ce^{-C't}.
\end{equation}
From \eqref{equation:thetaplusup}, \eqref{equation:e1} and \eqref{equation:e2}, it  follows that
\begin{equation}\label{equation:thetaplus1}
\Prob(\zeta_1>t^2, \ N^+ \ge 1, \ \tN^-=0) \le Ce^{-C't}.
\end{equation}

We will next address the case when $\tN^- \ge 1$. Note that this implies that $N^+\ge 1$. In this case the velocity $V$ first reaches level $-g+a_0+2$ and then $-g-a_0-2$ before the renewal time $\zeta_1$. Under this event, $V_{\talpha_0}=-g-a_0-2$ and $\zeta_1=\talpha_{3\tN^-}$. We need to control $\talpha_1-\talpha_0$ and $\talpha_2-\talpha_1$. At $\talpha_2$, we have $V_{\talpha_2}=-g-a_0-1$ and $S_{\talpha_2}=X_{\talpha_2}$ and we can apply the same analysis as in the case of the event $\left\{\vtau_{-g-a_0-2}< \vtau_{-g+a_0+2}\right\}$, replacing $\alpha$'s with $\talpha$'s and $N^-$ with $\tN^-$.

Note that for $\tN^- \ge 1$, $\talpha_0=\vtau_{-g-a_0-2}$ and $\talpha_1=\sigma(\vtau_{-g-a_0-2})$. For any $\delta, \eps>0$, we have
\begin{align}\label{align:downest}
&\Prob\left(\talpha_1-\talpha_0>\delta \sqrt{t}, \ \tN^- \ge 1 \right)\\
&\le \Prob\left(\talpha_1-\talpha_0>\delta \sqrt{t}, \ \talpha_0 \le \vtau_{\eps \sqrt{t}}  \wedge t^2, \ \tN^- \ge 1\right)
 + \Prob(\vtau_{\eps\sqrt{t}}< \talpha_0,  N^+ \ge 1) + \Prob(\talpha_0>t^2, \ \tN^- \ge 1).\nonumber
\end{align}

If we take $C_0 =1$, $a=\sqrt{t}$ and $m=4$ in Lemma \ref{lem:unifcontrol}
then we obtain for some $C_1$ and $C_2$, and all $t>0$,
\begin{equation*}
\Prob\left(\sigma(s) > s+3\frac{\delta \sqrt{t}}{\sqrt{g}} \text{ for some } 
s \le  t^2 \wedge \vtau_{-g+\sqrt{g}\delta \sqrt{t}/4}\wedge \vtau_{-g-\delta^2 t/8}
\mid V_0=0, H_0=0
\right) \le C_1 t^{3/2} e^{-C_2 t^{3/2}}.
\end{equation*}
Since $\talpha_0=\vtau_{-g-a_0-2}$, we have $\talpha_0\leq\vtau_{-g-\delta^2 t/8}$ for large $t$. Adjusting the values of the constants in the last estimate, we obtain that for any $\delta>0$, we can find $\eps>0$ such that
for sufficiently large $t$,
\begin{equation}\label{equation:sigmabeta}
\Prob\left(\talpha_1-\talpha_0>\delta \sqrt{t}, \ \talpha_0 \le \vtau_{\eps \sqrt{t}}  \wedge t^2, \ \tN^- \ge 1\right) \le Ct^{3/2}e^{-C't^{3/2}}.
\end{equation}

It follows from Theorem \ref{thm:fluconeside} that for sufficiently large $a>0$. 
\begin{align}\label{no3.1}
\Prob(\vtau_{a}<\vtau_{-g+1}\mid V_0=-g+a_0+2, H_0=0) \le Ce^{-C'a^2}.
\end{align}
Since Remark \ref{oc16.1} implies that $S_{\beta_{3k}} = X_{\beta_{3k}}$,
we can apply this inequality at $t=\beta_{3k}$, by the strong Markov property at $\beta_{3k}$. Note that $\sup_{s \in [\beta_{3k+1}, \beta_{3k+3}]} V_s
\leq -g+a_0+2$ so for large $a$,
$$
\left\{\sup_{s \in [\beta_{3k}, \beta_{3k+3}]} V_s >a \right\} = \left\{\sup_{s \in [\beta_{3k}, \beta_{3k+1}]} V_s >a \right\}.
$$
By  definition, 
$\beta_{3k+1} < \inf\{s\geq \beta_{3k}: V_s = -g+1\}$. These remarks
and \eqref{no3.1} imply that, for large $a$,
\begin{equation*}
\Prob\left(\sup_{s \in [\beta_{3k}, \beta_{3k+3}]} V_s >a, \ N^+\ge k+1 \right)
=\Prob\left(\sup_{s \in [\beta_{3k}, \beta_{3k+1}]} V_s >a, \ N^+\ge k+1 \right) \le Ce^{-C'a^2}.
\end{equation*}
This and \eqref{equation:estNplus} show that
 for  sufficiently large integer $a>0$,
\begin{align}
\Prob(\vtau_a< \beta_{3N^+}, \ N^+ \ge 1) &\le \sum_{k=0}^{a^2-1}\Prob\left(\sup_{s \in [\beta_{3k}, \beta_{3k+3}]} V_s >a, \ N^+\ge k+1 \right) + \Prob(N^+ > a^2) \nonumber \\
& \le Ca^2 e^{-C'a^2} + (p^+)^{a^2}\nonumber \\
& \le Ce^{-C'a^2}.\label{equation:uptillg}
\end{align}
Recall that on the event $\{N^+ \ge 1\}$, $V_u \in [-g,-g-a_0-2]$ for all $u \in [\beta_{3N^+},\talpha_0]$. Therefore, for $a>-g$, $$
\Prob(\vtau_a< \talpha_0, \ N^+ \ge 1)=\Prob(\vtau_a< \beta_{3N^+}, \ N^+ \ge 1).
$$
Thus, \eqref{equation:uptillg} with $a=\eps\sqrt{t}$  gives us, for sufficiently large $t>0$,
\begin{equation}\label{equation:uptillgt}
\Prob(\vtau_{\eps\sqrt{t}}< \talpha_0 , \ N^+ \ge 1) \le Ce^{-C'\eps^2t}.
\end{equation}

To estimate the last probability in \eqref{align:downest}, note that under the event $\{\tN^- \ge 1\}$,
$V_{\talpha_{-1}} \in [-g-a_0-2,-g)$ and $V_{\talpha_0}=-g-a_0-2$.  
Using \eqref{equation:thetaplusup}, \eqref{equation:e1} and \eqref{equation:e2}, we get
\begin{equation}\label{equation:talpha}
\Prob(\talpha_0>t^2, \ \tN^- \ge 1) \le Ce^{-C't}.
\end{equation}
Combining the estimates \eqref{equation:sigmabeta}, \eqref{equation:uptillgt} and \eqref{equation:talpha} with \eqref{align:downest}, we see that for any $\delta>0$ there is $t_0>0$ such that for $t \ge t_0$,
\begin{equation}\label{equation:talphagap}
\Prob\left(\talpha_1-\talpha_0>\delta \sqrt{t}, \ \tN^- \ge 1 \right) \le Ce^{-C't},
\end{equation}
where $C,C'$ may depend on $\delta$.

\medskip

We next control $\talpha_2-\talpha_1$. Write
\begin{equation}\label{equation:talphatwo}
\Prob(\talpha_2-\talpha_1 > t, \ \tN^- \ge 1) \le \Prob(\talpha_2-\talpha_1 > t, \talpha_1-\talpha_0 \le  \sqrt{t}, \ \tN^- \ge 1) + \Prob\left(\talpha_1-\talpha_0> \sqrt{t}, \ \tN^- \ge 1 \right).
\end{equation}
Applying the strong Markov property at $\talpha_1$ and then
Lemma \ref{lem:hittimeest}, we obtain for large $t$,
\begin{align}\label{align:talphathree}
\Prob&\left(\talpha_2-\talpha_1 > t, \talpha_1-\talpha_0 \le \sqrt{t}, \ \tN^- \ge 1\right)\nonumber\\
&\le \sup_{v \in [a_0+2, a_0+2+g\sqrt{t}]}\Prob(\vtau_{-g-a_0-1} >t \mid V_0=-g-v, H_0=0)\nonumber\\
& \le \sup_{v \in [a_0+2, a_0+2+g\sqrt{t}]}\frac{4(v-a_0-1)}{((v-a_0-1) +(a_0+1)t)(a_0+1)\sqrt{2\pi t}}e^{-(a_0+1)^2t/8} \le Ce^{-C't}.
\end{align}
Note that the hypotheses of Lemma \ref{lem:hittimeest}  hold 
because $t\ge 2(1+g\sqrt{t})/(a_0+1)$ for large $t$.

Substituting the estimates obtained in \eqref{equation:talphagap} (with $\delta=1$) and \eqref{align:talphathree} into \eqref{equation:talphatwo}, we get
\begin{equation}\label{equation:talphafour}
\Prob(\talpha_2-\talpha_1 > t, \ \tN^- \ge 1) \le Ce^{-C't}.
\end{equation}
The strong Markov property applied at $\talpha_2$ and \eqref{equation:hitg} imply that
$$
\Prob(\talpha_{3\tN^-} -\talpha_2>t^2, \ \tN^- \ge 1) \le \Prob(\vtau_{-g} > t^2 \mid V_0=-g-a_0-1, H_0=0) \le Ce^{-C't}.
$$
Thus,
\begin{align}\label{align:thetaone}
\Prob(\zeta_1-\talpha_2 >t^2, \ \tN^- \ge 1) &= \Prob(\talpha_{3\tN^-} -\talpha_2>t^2, \ \tN^- \ge 1) \le Ce^{-C't}.
\end{align}
Combining  \eqref{equation:talpha}, \eqref{equation:talphagap}, \eqref{equation:talphafour} and \eqref{align:thetaone}, we get
\begin{equation}\label{equation:thetaplus2}
\Prob(\zeta_1 >t^2, \ \tN^- \ge 1) \le Ce^{-C't}.
\end{equation}
From \eqref{equation:thetaplus1} and \eqref{equation:thetaplus2}, we obtain
\begin{equation} \label{equation:thetaplus}
\Prob(\zeta_1>t^2, \vtau_{-g+a_0+2}< \vtau_{-g-a_0-2}) \le Ce^{-C't}.
\end{equation}
The lemma follows from \eqref{equation:thetaminus} and \eqref{equation:thetaplus}.
\end{proof}

\section{The stationary distribution for $Z=(V,S-X)$}\label{sec:station}

We will use results from \cite{thorisson}, \cite{KaR14} to establish existence of a unique stationary distribution for $Z=(V, S-X)$ and convergence of $Z$ to the stationary distribution in total variation distance. For this reason, we will refer the reader to the book \cite{thorisson} even for definitions of widely use terms.
For example, the definition of the total variation distance can be found in 
\cite[Ch.~1, Sec.~5.3]{thorisson}. {We would like to point out that an alternative proof of  the existence  of the stationary distribution is contained in Corollary \ref{cor:stationary} and the calculations  in the proof of Theorem \ref{thm:stat}. Hence, the arguments in this section are mostly needed for the proof of uniqueness and convergence in total variation distance.}

The process $Z=(V, S-X)$ is a classical regenerative processes
(see \cite[Ch.~10, Sec.~3]{thorisson})
 with regeneration times $\zeta_k$ defined in Section \ref{sec:renewal}. This means that the process starts afresh at each $\zeta_k$ and the cycles of the process $Z$ given by
$
\mathcal{C}_k=(Z_t)_{t \in [\zeta_{k-1}, \zeta_k)}
$
for $k \ge 1$ form an i.i.d sequence. The gaps $\{\zeta_{k+1}-\zeta_k, k \ge 0\}$ between the regeneration times are called \textit{inter-regeneration times}. If the process starts from some  $z\in \mathbb{H}:= \Reals \times [0,\infty)$ then we will write $\Prob_z$ and $\EE_z$ for the law of $Z$  and the corresponding expectation. If we are not assuming that $Z$ starts from the renewal state, it will be convenient to redefine $\zeta_0$ as
\begin{align}\label{no12.1}
\zeta_0=\inf\{t \ge 0: Z_t=(-g,0)\}
\end{align}
and then shift the definition of the sequence $\zeta_k$ by $\zeta_0$.  The  arguments used in the proof of Lemma \ref{lem:estzeta} show that $\Prob(\zeta_0< \infty)=1$.
We will use $\Prob_R$ and $\Expect_R$ to denote probability and expectation when $Z$ starts from the renewal state, i.e., $Z_0=(-g,0)$.
The following lemma will be used in the proof of Theorem \ref{thm:stationary}.

\begin{lem}\label{lem:density}
The inter-regeneration time $\zeta_1-\zeta_0$ has a density with respect to Lebesgue measure.
\end{lem}

\begin{proof}

Recall $\eta_0$  defined in \eqref{align:eta} and the notation $\HH=\Reals\times\Reals_+$.

By the strong Markov property applied at $\zeta_0$, for any $z \in \HH$ and any measurable set $\A  \subseteq [0, \infty)$, $\Prob_z(\zeta_1-\zeta_0 \in \A )=\Prob_R(\zeta_1 \in \A )$. Thus, it is enough to prove that starting from $V_0=-g, H_0=0$, the random variable $\zeta_1$ has a density.  Note that $\zeta_1\geq\eta_0 >0$, a.s., under $\PP_R$. By the Radon-Nikodym Theorem it suffices to show that $\Prob_R(\zeta_1 \in \A )=0$ for any $b<\infty$ and any  set $\A  \subseteq [0,b]$ of Lebesgue measure zero. 

Consider the stopping time
$$
T=\inf\{t \ge \eta_0: V_t \le -g, \ S_t=X_t\}.
$$
Since $T\leq \zeta_1$ a.s., by Remark \ref{oc16.1}    we have that $\zeta_1 = T + \vtau_{-g}\circ \theta_T$, where $\theta$ is the standard shift operator for Markov processes. Applying the strong Markov property  at  $T$ we obtain
\begin{align}
\label{eq:density_zeta_01}
\PP_R(\zeta_1\in A) &= \int_{-\infty}^{-g}  \int_0^{b}\PP( \vtau_{-g}\in A-t \mid V_0=v,\ H_0=0)\PP_R(T\in dt,\ V_T\in dv).
\end{align}

Suppose that $V_0=v<-g$. Note that 
\begin{align*}
\vtau_{-g}&=\inf\left\{t \ge 0: \sup_{u \le t}(B_u-S_u) - gt=-g-v\right\}
=\inf\left\{ t \ge 0: B_t - S_t -gt =-g-v\right\}\\
&=\inf\left\{ t \ge 0: B_t - \int_0^t(V_u + g)du=-g-v\right\},
\end{align*}
Consider  Brownian motion with drift 
\begin{align*}
W_t = B_t - \int_0^t \left(V\left(u\wedge \vtau_{-g}\wedge \vtau_{v-g\sup A}\right)+g\right)du,
\end{align*}
 and let $\tau^W_a$ be the hitting time of $a$ by $W$. The last displayed formula shows that $\kpar{\vtau_{-g}\in A-t} = \kpar{\tau^W_{-g-v}\in A-t}$. Since the drift of $W$ is bounded on finite time intervals, the Girsanov theorem implies that the laws of $W$ and $B$ are mutually absolutely continuous on finite time intervals. The event  $\kpar{\tau^{B,0}_{-g-v}\in A-t}\in \F_{\sup A}\subseteq\F_b$ has probability zero since the law of $\tau^{B,0}_{-g-v}$ is absolutely continuous, and $A-t$ has zero Lebesgue measure. Thus, the event $\kpar{\tau^{W}_{-g-v}\in A-t}$ has probability zero. We conclude that  $\PP\rpar{\vtau_{-g-v}\in A-t \mid V_0=v,\ H_0=0} = 0$ for any $v<-g$.

The discussion above allows us to transform \eqref{eq:density_zeta_01} into
\begin{align}
\label{eq:density_zeta_02}
\PP_R(\zeta_1\in A) &= \int_0^{b}\Indicator_{A}(t)\PP_R(T\in dt,\ V_T=-g)\leq \PP_R(V_T = -g) = \PP_R(T=\zeta_1).
\end{align}

 We will show that $T < \zeta_1$, a.s. This  holds if $V_{\eta_0}=-g-a_0-2$ because then $V_t \leq -g-a_0-2$ for $t\in [\eta_0, T]$. Suppose that $V_{\eta_0}=-g+a_0+2$ and note that $S_{\eta_0} = X_{\eta_0}$, by Remark \ref{oc16.1}. Hence, by the strong Markov property, it is enough to show that starting from $V_0=-g+a_0+2$ and $ H_0=0$, we have $T=\sigma(\vtau_{-g})< \zeta_0$, a.s. (recall \eqref{no12.1}). We will show that, with probability one, $B_{\vtau_{-g}}-S_{\vtau_{-g}} < \sup_{u \le \vtau_{-g}}(B_u-S_u)$, or, equivalently, in view of \eqref{oc9.1},
$X_{\vtau_{-g}}<S_{\vtau_{-g}}$. This shows that, a.s., $\sigma(\vtau_{-g})> \vtau_{-g}$
and, therefore, $V_{\sigma(\vtau_{-g})}<-g$, implying that $\sigma(\vtau_{-g})< \zeta_0$.

For any  $r > -g+a_0+2$, the process $V^*_t:= V_{t \wedge \vtau_{-g} \wedge \vtau_r}$ is bounded and thus the exponential local martingale
\begin{align}\label{no13.1}
 t \mapsto \exp\left\lbrace \int_0^t V^*_u dB_u - \frac{1}{2}\int_0^t (V^*_u)^2du\right\rbrace
\end{align}
is a positive martingale for all $t \ge 0$. 
Hence, by the Girsanov Theorem, the laws of processes $B_t - S_t$ and $B_t$ are mutually absolutely continuous on finite time intervals. For $t\geq 0$, let us define
\begin{align*}
M_t=\sup_{u \le t}B_u,\qquad \tau^{M,-g}_a = \inf\kpar{ t\geq 0 : M_t -gt = a}.
\end{align*}
The  mutual absolute continuity of the laws of 
$B_t - S_t$ and $B_t$ implies that if
\begin{equation}\label{equation:girtwo}
\Prob\left(B_{\tau^{M,-g}_{-a_0-2}}= \sup_{u \le \tau^{M,-g}_{-a_0-2}}B_u\right)=0,
\end{equation}
then
\begin{equation}\label{equation:girone}
\Prob\left(B_{\vtau_{-g}}-S_{\vtau_{-g}} = \sup_{u \le \vtau_{-g}}(B_u-S_u), \ \vtau_{-g} \le \vtau_r
\mid V_0 = -g+a_0+2,H_0=0
\right)=0.
\end{equation}
We will prove that \eqref{equation:girtwo} is  true. We need  more definitions. Set
\begin{align*}
M^{-1}(t)&= \inf\{u>0: M_u > t\},\\
\lambda(x) &=  \inf\{ t \ge 0: M^{-1}(t) - t/g \ge x\}.
\end{align*}
Note that
\begin{equation}\label{equation:inv}
g \tau^{M,-g}_{-a_0-2}-a_0-2 = \lambda((a_0+2)/g).
\end{equation}

Now we will use some facts from the theory of Levy processes. 
It is well known that $M^{-1}_t$ is a stable process with index $1/2$. By
\cite[Thm.~5(i), Ch.~VIII]{Bert}, we have $\limsup_{t\downarrow 0} M^{-1}(t) t^{2-\eps} = 0$, a.s., for every $\eps >0$.
It follows that
0 is an irregular point for $[0, \infty)$, i.e.,
if $M^{-1}(0) =0$ then
\begin{align*}
\Prob(\inf\{t>0: M^{-1}(t) -t/g \in [0, \infty)\}=0)=0.
\end{align*}
This fact implies that the ascending ladder height process $K$, which can roughly be viewed as the values taken by $M^{-1}(t) -t/g$ at its successive new maxima (see  \cite[Sec.~6.2]{kyprianou} for a more precise definition), is a compound Poisson subordinator (see the last paragraph on p. 150 of \cite{kyprianou}). This implies that the jump distribution of $K$ does not have atoms and thus, a.s., $K$ does not hit specified points (see the last paragraph in the proof of \cite[Lem.~7.10]{kyprianou}). Thus, a.s., 
$$M^{-1}(\lambda((a_0+2)/g))-\frac{\lambda((a_0+2)/g)}{g} >(a_0+2)/g,
$$
which, by the definition of $\tau^{M,-g}_{-a_0-2}$ and \eqref{equation:inv}, gives, a.s.,
$$
M^{-1}\left(M_{\tau^{M,-g}_{-a_0-2}}\right)=
M^{-1}\left(g\tau^{M,-g}_{-a_0-2} -a_0-2\right) > \tau^{M,-g}_{-a_0-2}.
$$
This means that at the stopping time $\tau^{M,-g}_{-a_0-2}$,  Brownian motion $B$ lies strictly below its running maximum, which gives \eqref{equation:girtwo} and thus \eqref{equation:girone}. By letting $r \rightarrow \infty$ in \eqref{equation:girone}, we get $B_{\vtau_{-g}}-S_{\vtau_{-g}} < \sup_{u \le \vtau_{-g}}(B_u-S_u)$, a.s., which in turn implies that $T <\zeta_1$, a.s., in view of the
 opening remarks of the proof.
\end{proof}

Recall that $Z=(V,S-X)$ has the state space $\mathbb{H}= \Reals \times [0,\infty)$.

\begin{thm}\label{thm:stationary}
For any $z\in\mathbb{H}$, assuming that $Z_0=z$,
when $t \rightarrow \infty$, the law of $Z_t$ converges in total variation distance to a unique stationary distribution $\pi$, given by 
\begin{align}\label{eq:stationarypi}
\pi(A) = \frac{ \Expect_R\left( \int_{0}^{\zeta_1} \Indicator_A(Z_u) du \right) }{\Expect_R (\zeta_1)},
\end{align}
for any measurable set $A\subseteq \mathbb{H}$. Furthermore,
for every $z \in \mathbb{H}$, $\PP_z$-a.s.,
\begin{align}\label{align:almostsure}
\lim_{t \rightarrow \infty}\frac{\int_0^t \Indicator_A(Z_u)du}{t} = \pi(A).
\end{align} 
\end{thm}

\begin{proof}
Recall the regeneration times $\zeta_k$ from the beginning of this section.
Lemma \ref{lem:estzeta} shows that $\Expect(\zeta_1-\zeta_0) < \infty$. By Lemma \ref{lem:density}, the distribution of an inter-regeneration time $\zeta_{k+1}-\zeta_k$ is spread out in the sense of \cite[Ch.~10, Sec.~3.5]{thorisson}. The convergence in total variation distance and, hence, the uniqueness of the stationary distribution now follows from part (b) of  \cite[Ch.~10, Thm.~3.3]{thorisson} (see Section 2.4 in that chapter for the notation used in the cited theorem). The representation \eqref{eq:stationarypi} follows from (2.1) in \cite[Ch.~10, Thm.~2.1]{thorisson}.

Let
$
N_t = \sup \{k \ge 0: \zeta_k\leq t \}
$
be the number of renewals up to time $t$. The  arguments applied in the proof of Lemma \ref{lem:estzeta} can be used to show that $\Prob(\zeta_0< \infty)=1$.
To prove \eqref{align:almostsure}, note that for any measurable  set $A \subseteq \mathbb{H}$, we have
\begin{align*}
\int_0^t \Indicator_A( Z_u) du &= \int_0^{\zeta_0} \Indicator_A( Z_u) du + \sum_{k=1}^{N_t}\int_{\zeta_{k-1}}^{\zeta_k} \Indicator_A( Z_u) du + \int_{\zeta_{N_t}}^{t} \Indicator_A( Z_u) du.
\end{align*}
Clearly $(1/t)\int_0^{\zeta_0} \Indicator_A( Z_u) du \rightarrow 0$, a.s., as $t \rightarrow \infty$. Furthermore, $(1/t)\int_{\zeta_{N_t}}^{t} \Indicator_A( Z_u) du \le (1/t) (t-\zeta_{N_t}) \rightarrow 0$, a.s.,  by  \cite[Prop.~7.3]{ross}. The same proposition implies that
\begin{align*}
\frac{\sum_{k=1}^{N_t}\int_{\zeta_{k-1}}^{\zeta_k} \Indicator_A( Z_u) du}{t} \rightarrow \frac{ \Expect_R\left( \int_{0}^{\zeta_1} \Indicator_A(Z_u) du \right) }{\Expect_R (\zeta_1)},
\end{align*}
almost surely as $t \rightarrow \infty$, so \eqref{align:almostsure}
follows from \eqref{eq:stationarypi}.
\end{proof}

The following corollary to Lemma \ref{lem:submg} shows that we can use \cite[Thm.~1]{KaR14} in our context, and will be the main tool in the proof of Theorem \ref{thm:stat}. Recall that $\HH=\Reals\times\Reals_+$.

\begin{cor}
\label{cor:stationary}
Let $\pi$ be a probability measure on $(\HH,\B(\HH))$, with $\pi(\partial\HH)=0$. Then, $\pi$ is a stationary distribution for the solution to \eqref{eq:z} if and only if, for every $f\in C^2_c(\overline{\HH})$ such that $(\nabla f(z))^T\gamma\geq 0$ for all $z\in\partial\HH$, the inequality
\begin{align}
\label{eq:stationary}
\int_{\HH} \L f(z)\pi(dz) \leq 0.
\end{align}
holds.

\end{cor}
\begin{proof}
By Lemma \ref{lem:submg}, the law of the solution to  \eqref{eq:z} and the solution to the submartingale problem for $(\L,\gamma)$ in $\HH$ are the same. We only need to show that \cite[Thm.~1]{KaR14} applies.

Note that for any constant $c\in\Reals$ we have $\L c=0$, thus, the set $C^2_c(\overline{\HH})$ and the set $\mathcal{H}$ from \cite{KaR14} coincide. Also, the set $\mathcal{V}$ from \cite{KaR14} is empty in our case. We next check that Assumption 1 in \cite{KaR14} holds.
\begin{enumerate}
\item Routine arguments show that $C^2_c(\overline{\HH})$ separates points.

\item 
We will use the notation from  \cite{KaR14},
including but not limited to Assumption 1.
Since our reflection field $\gamma=(1,1)^T$ is constant,  we have $d(x)\cap S_1(0)=\frac1{\sqrt 2}\gamma$. The function $f(v,h) = v+h$ satisfies $\left(\nabla f(v,h)\right)^T \frac1{\sqrt 2}\gamma=1$. Let $\eta \in C^\infty(\overline\HH)$ be compactly supported and equal to 1 on an open neighborhood of $x$.
The function $f_{r,s} := f \eta$ 
satisfies the condition stated in Assumption 1 of \cite{KaR14}.
\end{enumerate}

This shows that Assumption 1 holds so  \cite[Thm.~1]{KaR14} applies in our case.
\end{proof}

\begin{proof}[Proof of Theorem \ref{thm:stat}]

The existence and uniqueness of the stationary distribution and convergence of $Z_t$ to this distribution in total variation distance were proved in Theorem \ref{thm:stationary}.

We will apply Corollary \ref{cor:stationary}. It is readily verifiable from \eqref{eq:density} that the proposed stationary distribution $\pi(dv,dh)=\xi(v,h)dvdh$ is indeed a probability distribution, and as it has a density, therefore trivially, we have $\pi(\partial \mathbb{H})=0$. We next show that \eqref{eq:stationary} holds.

Recall that $\mathcal{L}f(v,h)=\frac12 \partial_{hh}f +v\partial_h f -g\partial_v f $, and $\gamma=(1,1)^T$. Let $f\in C^2_c(\overline{\mathbb{H}})$ be a function satisfying
\begin{align}
\label{eq:bdry_condition}
(\nabla f(v,0))^T\gamma  =\partial_h f(v,0)+\partial_v f(v,0) \geq 0
\end{align}
for $v\in\Reals$. We proceed by direct computation. By Fubini's Theorem,
\begin{align*}
\int_{\mathbb{H}}  \mathcal{L} f\ \xi(v,h)dvdh &= \frac{2g}{\sqrt{\pi}} \int_0^\infty \int_{-\infty}^\infty \left( \frac12 \partial_{hh}f +v\partial_h f -g\partial_v f \right) e^{-2gh}e^{-(v+g)^2} dvdh \\
&= -\frac{2g^2}{\sqrt{\pi}} \int_0^\infty \int_{-\infty}^\infty  \partial_v f(v,h)  e^{-(v+g)^2} dv\  e^{-2gh}dh  \\
&\qquad +\frac{2g}{\sqrt{\pi}} \int_{-\infty}^\infty  \int_0^\infty  \left( \frac12 \partial_{hh}f +v\partial_h f \right) e^{-2gh} dh \ e^{-(v+g)^2}dv.
\end{align*}
Next, we integrate by parts the inner integral of each term; actually, we have to integrate by parts twice. We get
\begin{align*}
\int_{\mathbb{H}}  \mathcal{L} f\ \xi(v,h)dvdh 
&= -\frac{4g^2}{\sqrt{\pi}} \int_0^\infty  \int_{-\infty}^\infty   f(v,h)(v+g)  e^{-(v+g)^2} dv\ e^{-2gh} dh  \\
&\qquad -\frac{2g}{\sqrt{\pi}} \int_{-\infty}^\infty  \left( \frac12\partial_hf(v,0) +(v+g)f(v,0)  \right) e^{-(v+g)^2}dv\\
&\qquad+\frac{2g}{\sqrt{\pi}} \int_{-\infty}^\infty \int_0^\infty  \left( 2g^2  +2vg \right)f(v,h) e^{-2gh} dh \,e^{-(v+g)^2}dv.
\end{align*}
The two double integrals above cancel each other so we are left with a single integral,
\begin{align*}
\int_{\mathbb{H}}  \mathcal{L} f\ \xi(v,h)dvdh 
&= -\frac{2g}{\sqrt{\pi}} \int_{-\infty}^\infty \left(  \frac12\partial_hf(v,0) +(v+g)f(v,0)  \right) e^{-(v+g)^2}dv\\
&= -\frac{2g}{\sqrt{\pi}} \int_{-\infty}^\infty  \left(\frac12\partial_hf(v,0) e^{-(v+g)^2} -\frac12 f(v,0)\partial_v e^{-(v+g)^2}\right)dv.
\end{align*}
Integrating by parts once again we obtain
\begin{align*}
\int_{\mathbb{H}}  \mathcal{L} f\ \xi(v,h)dvdh 
&= -\frac{g}{\sqrt{\pi}} \int_{-\infty}^\infty  \left(\partial_hf(v,0) e^{-(v+g)^2} + \partial_vf(v,0) e^{-(v+g)^2}\right)dv \\
&= -\frac{g}{\sqrt{\pi}} \int_{-\infty}^\infty  \left( \partial_hf(v,0)  + \partial_vf(v,0) \right) e^{-(v+g)^2}dv \\
&= -\frac{g}{\sqrt{\pi}} \int_{-\infty}^\infty (\nabla f(v,0))^T\gamma \ e^{-(v+g)^2}dv,
\end{align*}
which is nonpositive by \eqref{eq:bdry_condition}.
\end{proof}

\begin{remark}\label{rem:uniquest}
Recall that we showed existence and uniqueness of the stationary distribution in Theorem \ref{thm:stationary}. The existence is also implied by \cite[Thm.~1]{KaR14} via Corollary \ref{cor:stationary} since we showed that \eqref{eq:stationary} holds for the distribution $\pi$ with density (with respect to Lebesgue measure) given by \eqref{eq:density}. We also believe that the uniqueness of the stationary distribution follows from the Harris irreducibility of our process. But the main power of Theorem \ref{thm:stationary} comes from the fact that it uses Lemma \ref{lem:estzeta} and Lemma \ref{lem:density} to show convergence to stationarity in total variation distance (the existence and uniqueness of the stationary distribution follow as a by-product of this convergence). 
\end{remark}

\section{Fluctuations and Strong Laws of Large Numbers}\label{sec:strong}

{This section is dedicated to the proofs of Theorem \ref{thm:velfluc} and Theorem \ref{thm:stronglaw}. We first derive fine estimates for a number of hitting times of the velocity process and the gap process. These, in turn, are used to establish precise estimates for the fluctuations of $V$ and $S-X$ on the intervals $[\zeta_n, \zeta_{n+1}]$ for $n \ge 0$, where $(\zeta_n)_{n \ge 0}$ are the renewal times defined in Section \ref{sec:renewal}. The global fluctuation results of Theorem \ref{thm:velfluc} and Theorem \ref{thm:stronglaw} are then proved by decomposing the path $\{Z_s: s \le t\}$ into random time intervals $[\zeta_n, \zeta_{n+1}]$ and applying the fluctuation results proved on these intervals.} 
\begin{lem}\label{no5.5}

(i) For some $C>0$, all $0\leq h \leq 1$ and $a\leq -g-2$,
\begin{align*}
\EE\left(\vtau_{-g} \mid V_0 = a , H_0 = h\right) \leq C|a|.
\end{align*}

(ii) For some $C>0$, all $h\geq 0$ and $a\in\Reals$,
\begin{align*}
\EE\left(\vtau_{a+1} \mid V_0 = a , H_0 = h\right) \geq  
C\frac 1 {(-a)\lor 1}.
\end{align*}

(iii) For some $C>0$, all $h \geq 0$ and $a\geq 2$, 
\begin{align*}
\EE\left(\vtau_{-g} \mid V_0 = a , H_0 = h\right) \leq C a.
\end{align*}

(iv) For any $h\geq 0$ and $a\in \Reals$,
\begin{align*}
\EE\left(\vtau_{a-1} \mid V_0 = a , H_0 = h\right) \geq 1/g.
\end{align*}

(v) For some $C>0$, any $h\geq 1$ and $-g-2\leq a \leq -g-1$,
\begin{align*}
\EE\left(\vtau_{-g}  
\mid V_0 = a , H_0 = h\right) \leq  C\sqrt{h}.
\end{align*}

(vi) For some $C>0$, any $h\geq 1$ and $a\in \Reals$,
\begin{align*}
\EE\left(\stau_{h-1} \land \vtau_{a-1} 
\mid V_0 = a , H_0 = h\right) \geq  C \frac 1 {(-a)\lor 1} .
\end{align*}

\end{lem}

\begin{proof}

(i) If $V_0=a\leq -g-2$, then  $\vtau_{-g} = \sigma(0) + \vtau_{-g}\circ\theta_{\sigma(0)}$ where $\theta$ is the standard shift operator for Markov processes. 
By applying the strong Markov property at time $\sigma(0)$  we obtain
\begin{align}\label{f4.1}
&\EE(\vtau_{-g} \mid V_0=a, H_0=h) \\
&\le
\EE (\sigma(0) \mid V_0=a, H_0 = h) 
+ \sum_{k=0}^{\infty} \EE(\vtau_{-g} \mid V_0 = a-k-1, H_0=0) \PP(V_{\sigma(0)} \leq a-k \mid V_0=a,\ H_0=0)\nonumber\\
&\le
\EE (\sigma(0) \mid V_0=a, H_0 = h) 
+ \sum_{k=0}^{\infty} \EE(\vtau_{-g} \mid V_0 = a-k-1, H_0=0) \PP( \sigma(0) \geq k/g \mid V_0=a,\ H_0=0). \nonumber
\end{align}
We will next estimate the terms in the sum.

Suppose that $B_0=0$, $a\leq 0$, $V_0 = a$, $h\geq 0$ and $H_0 = h$. If
$\sigma(0) \geq s$ then $B_t \leq h + a t - gt^2/2 \leq h - gt^2/2$
for all $t\leq s$. In particular, $B_s \leq h - g s^2/2$. Thus, using \eqref{oc13.3}, for $s\geq 2\sqrt{h/g}$,
\begin{align}\label{no7.4}
\PP(\sigma(0) \geq s  \mid V_0=a, H_0 = h) \leq 
\PP(B_s \leq h - g s^2/2)
\leq C \frac{\sqrt{s}}{{g s^2-2h}} e^{-(h - g s^2/2)^2/(2s)}.
\end{align}
This implies the following two bounds, with $0< C, C_1<\infty$,
\begin{align}
\label{no7.3}
\EE (\sigma(0) \mid V_0=a, H_0 = h) &\leq C,\qquad 0\leq h\leq 1 \\
\label{no7.5}
\PP(\sigma(0) \ge k/g \mid V_0=a,\ H_0=h) &\leq C k^{-1}e^{-C_1 k^3},\qquad k\geq 1.
\end{align}

Next, we bound the expectations in \eqref{f4.1}. Consider $b \leq - g-2$. According to  
Lemma \ref{lem:hittimeest} applied with $-g-a_1 = b+1$ and $-g-a_2 = b$,
for $t\ge 2/(-g-b-1) $,
\begin{equation*}
\Prob\left(\vtau_{b+1}>t \mid V_0=b, H_0=0\right) \le \frac{4}
{(1+(-g-b-1)t)(-g-b-1)\sqrt{2\pi t}}e^{-(-g-b-1)^2t/8}.
\end{equation*}
Since $-g-b-1 \geq 1$, we obtain for $t\geq 2$,
\begin{align*}
\Prob(\vtau_{b+1}>t \mid V_0=b, H_0=0) \le \frac{4}
{(1+t)\sqrt{2\pi t}}e^{-t/8}.
\end{align*}
Hence,
\begin{align*}
\EE(\vtau_{b+1} \mid V_0=b, H_0=0) \le C.
\end{align*}
Recall from Remark \ref{oc16.1} that if $V_0 = b-k$ then
$S_{\vtau_{b-k+1}} = X_{\vtau_{b-k+1}}$. Hence, the repeated application of the last estimate at the stopping times $\vtau_{b-k+1}, \vtau_{b-k+2}, \vtau_{b-k+3}, \dots,$ shows that if $b \leq -2g - 1$ then
\begin{align}\label{no7.1}
\EE(\vtau_{b} \mid V_0=b-k, H_0=0) \le Ck.
\end{align}

We can take $a_0 = g + 1$ in
\eqref{equation:hitg} to obtain for $-2g-2 \leq b \leq -2g-1$,
\begin{align*}
\EE(\vtau_{-g} \mid V_0=b, H_0=0) \le
\EE(\vtau_{-g} \mid V_0=-2g-2, H_0=0) \le C.
\end{align*}
This, the strong Markov property applied at $\vtau_{b}$ where $b$ satisfies $-2g-2\leq b \leq -2g-1$, and \eqref{no7.1}  imply that for $a \leq -g-2$,
\begin{align}\label{no7.2}
\EE(\vtau_{-g} \mid V_0=a, H_0=0) \le C(-g-a).
\end{align}

Substituting \eqref{no7.3}, \eqref{no7.5} and \eqref{no7.2} into \eqref{f4.1}, we get
\begin{align*}
\EE(\vtau_{-g} \mid V_0=a, H_0=h) 
&\leq C + \sum_{k=1}^\infty C (-g-a+k+1) k^{-1} e^{-C_1 k^3} \leq C\abs{a}.
\end{align*} 
 
\medskip

(ii)
It is easy to see that for all $h\geq 0$,
\begin{align*}
\EE\left(\vtau_{a+1} \mid V_0 = a ,  H_0 = h\right) \geq 
\EE\left(\vtau_{a+1} \mid V_0 = a , H_0 = 0\right),
\end{align*}
so we will assume that $V_0=a$ and $H_0=0$. 

For $t \leq 1/g$, we have $V_t \geq a-1$, so
 $\displaystyle{L_t \le L^{(a-1)}_t=\sup_{u\le t}(B_u-(a-1)u)}$. 
Suppose that $a\leq -g/2+1$ and, therefore, $0\leq -1/(2(a-1)) \leq 1/g$.
If $\sup_{u\leq -1/(2(a-1))} B_u < 1/2$ then,
\begin{align*}
L_{-1/(2(a-1))} \le L^{(a-1)}_{-1/(2(a-1))}
=\sup_{u\le -1/(2(a-1))}(B_u-(a-1)u)
\leq \sup_{u\le -1/(2(a-1))}B_u + \frac{a-1}{2(a-1)} < 1,
\end{align*}
and, therefore, $\vtau_{a+1} \geq -1/(2(a-1))$. It follows  that,
\begin{align*}
\EE&\left(\vtau_{a+1} \mid V_0 = a , S_0 - X_0 = 0\right)
\geq \frac 1 {-2 (a-1)} \PP\left(\vtau_{a+1} \geq \frac 1 {-2 (a-1)} \right)\\
&\geq \frac 1 {-2 (a-1)} \PP\left(\sup_{u\leq -1/(2(a-1))} B_u < 1/2 \right)
\geq C/|a|.
\end{align*}

It is easy to check that for every $a\geq -g/2+1$, the same argument
generates a bound at least as large as for $a= -g/2+1$. This easily translates onto the statement in part (ii) of the lemma.

\medskip

(iii) Consider $a \geq 2$ and take any $ v \in (a-1,a]$.
According to  
Lemma \ref{lem:hitup} (i) applied with $-g+a_1 = a-1$ and $-g+a_2 = v$,
for $t\ge 2/g $,
\begin{equation*}
\Prob(\vtau_{a-1}>t \mid V_0=v, H_0=0) \le e^{-g(a-1)t}
\leq e^{-gt}.
\end{equation*}
Hence,
\begin{align}\label{align:uni}
\sup_{v \in (a-1,a]}\EE(\vtau_{a-1} \mid V_0=v, H_0=0) \le C.
\end{align}
We can write $\vtau_{a-1} = \sigma(0) \wedge \vtau_{a-1} + \vtau_{a-1} \circ \theta_{\sigma(0) \wedge \vtau_{a-1}}$ where $\theta$ is the standard shift map. Therefore, by the strong Markov property applied at the stopping time $\sigma(0)$, we get for any $h \ge 0$,
\begin{align*}
\EE(\vtau_{a-1} \mid V_0=a, H_0=h) &= \EE(\sigma(0) \wedge \vtau_{a-1} + \vtau_{a-1} \circ \theta_{\sigma(0) \wedge \vtau_{a-1}} \mid V_0=a, H_0=h)\\
&\le \EE(\sigma(0) \wedge \vtau_{a-1}\mid V_0=a, H_0=h) + \sup_{v \in (a-1,a]}\EE(\vtau_{a-1} \mid V_0=v, H_0=0).
\end{align*}
Note that if $\sigma(0) >1/g$, then $\vtau_{a-1}=1/g$. Thus, $\EE(\sigma(0) \wedge \vtau_{a-1}\mid V_0=a, H_0=h) \le 1/g$. Applying this and \eqref{align:uni} to the above inequality yields
\begin{align*}
\sup_{h \ge 0} \ \EE(\vtau_{a-1} \mid V_0=a, H_0=h) \le C.
\end{align*}
The repeated application of the last estimate at the stopping times $\vtau_{a-1}, \vtau_{a-2}, \vtau_{a-3}, \dots,$ shows that if $a- k \geq 1$ and $h\geq 0$ then
\begin{align}\label{no8.2}
\EE(\vtau_{a-k} \mid V_0=a, H_0=h) \le Ck.
\end{align}
For $1 \leq a \leq 2$ and $h \ge 0$, we can write $\vtau_{-g} = \sigma(0) \wedge \vtau_{-g} + \vtau_{-g} \circ \theta_{\sigma(0) \wedge \vtau_{-g}}$. It is clear that $\sigma(0) \wedge \vtau_{-g} \le (g+2)/g$. Thus, we get
\begin{equation}\label{eq:negone}
\EE(\vtau_{-g} \mid V_0=a, H_0=h) \le (g+2)/g + \sup_{v \in (-g,2]}\EE(\vtau_{-g} \mid V_0=v, H_0=0).
\end{equation}
By another application of the strong Markov property, we can write
\begin{equation*}
\sup_{v \in (-g,2]}\EE(\vtau_{-g} \mid V_0=v, H_0=0) \le \sup_{v \in (-g,2]}\EE(\vtau_{-g} \wedge \vtau_2 \mid V_0=v, H_0=0) + \EE(\vtau_{-g} \mid V_0=2, H_0=0).
\end{equation*}
Applying Lemma \ref{lem:tube}, we get $\displaystyle{\sup_{v \in (-g,2]}\EE(\vtau_{-g} \wedge \vtau_2 \mid V_0=v, H_0=0) \le C}$. Further, taking $a_0 = g + 1$ in \eqref{no8.1}, we obtain $\EE(\vtau_{-g} \mid V_0=2, H_0=0) \le C$. Substituting these estimates into \eqref{eq:negone}, we get
$$
\sup_{a \in [1,2], \ h \ge 0}\EE(\vtau_{-g} \mid V_0=a, H_0=h) \le C.
$$
This, the strong Markov property applied at $\vtau_{a-k}$ where $k$ satisfies $1\leq a-k \leq 2$, and \eqref{no8.2}  imply that for $a \geq 2$ and $h\geq 0$,
\begin{align*}
\EE(\vtau_{-g} \mid V_0=a, H_0=h) \le C a.
\end{align*}

\medskip

(iv) The estimate follows from the equation $dV_t = dL_t -g dt$ and the fact that $L$ is non-decreasing.

\medskip

(v) Assume that $V_0=a$, and $H_0=h\geq 1$. By \eqref{no7.4}, we have  for $s\geq 2/\sqrt{g}$:
\begin{align*}
\PP(\sigma(0) \ge s\sqrt{h}) \leq C \frac{h^{1/4}\sqrt{s}}{ gh s^2/2-h} e^{-(h - gh s^2/2)^2/(2s\sqrt{h})}
\leq
C \frac{\sqrt{s}}{ g s^2/2-1} e^{-(1 - g s^2/2)^2/(2s)}.
\end{align*}
This implies that 
\begin{align}\label{no8.11}
\EE (\sigma(0) \mid V_0=a, H_0 = h) \leq C\sqrt{h}.
\end{align}
Note that $V_{\sigma(0)} = a -g\sigma(0)$ so
\eqref{no7.4} yields for $k\geq 2\sqrt{hg}$,
\begin{align*}
\PP(V_{\sigma(0)} \in [a-k-1, a-k])
&\leq \PP(V_{\sigma(0)} \le a-k)
=\PP(\sigma(0) \geq k/g) 
\\
&\leq C \frac{\sqrt{k/g}}{ g (k/g)^2-2h} e^{-(h - g (k/g)^2/2)^2/(2k/g)}.
\end{align*}
This, the strong Markov property applied at $\sigma(0)$, \eqref{no7.2} and
\eqref{no8.11} imply that
\begin{align*}
\EE&(\vtau_{-g} \mid V_0=a, H_0=h) \\
&\le
\EE (\sigma(0) \mid V_0=a, H_0 = h) 
+ \sum_{k=0}^\infty C(-g-a+k+1) \PP(V_{\sigma(0)} \in [a-k-1, a-k])\\
&\leq C\sqrt{h} +  
(-g-a+2\sqrt{hg}+1)\\
&\qquad+ \sum_{k=\lfloor2\sqrt{hg}\rfloor+1}^\infty C(-g-a+k+1)
\frac{\sqrt{k/g}}{g (k/g)^2-2} e^{-(1 - g (k/g)^2/2)^2/(2k/g)}\\
&\leq C\sqrt{h} +  
(-g+g+2+2\sqrt{hg}+1)\\
&\qquad+ \sum_{k=\lfloor2\sqrt{hg}\rfloor+1}^\infty C(-g+g+2+k+1)
\frac{\sqrt{k/g}}{ g (k/g)^2-2} e^{-(1 - g (k/g)^2/2)^2/(2k/g)}\\
&\leq C \sqrt{h}.
\end{align*}

\medskip

(vi) 
Suppose that $S_0=h\geq 1$, $X_0=0$ and $V_0=a\leq -\sqrt{g/8}$. Then, for $0\leq t \leq - 1/(4a)$,
\begin{align*}
S_t \geq h-gt^2/2 + at  \geq 
h - \frac g 2 \cdot \frac 1 {16 a^2} - a \frac 1 {4a} 
\geq
h-1/2.
\end{align*}
Therefore, for any $t_0 \in (0,-1/(4a)]$, if $\sup_{u\leq t_0} B_u \leq 1/2$, then $X_t \leq 1/2$ 
for $0\leq t \leq t_0$
and, therefore, $\stau_{h-1} \geq t_0$. 
The local time $L$ does not increase on the interval
$[0,\stau_{h-1}]$ because $S$ and $X$ do not meet on this interval. Hence, $V$ decreases at the constant rate on this interval and, therefore,
$\stau_{h-1} \land \vtau_{a-1}
= \stau_{h-1} \land 1/g$.
Thus,
\begin{align*}
&\PP\left(\stau_{h-1} \land \vtau_{a-1} \geq \frac{1}{(-4a) \lor g}\right)
=\PP\left(\stau_{h-1} \land 1/g  \geq \frac{1}{(-4a) \lor g}\right) =
\PP\left(\stau_{h-1}   \geq \frac{1}{(-4a) \lor g}\right)\\
&\geq
\PP\left(\sup_{u\leq 1/((-4a) \lor g)} B_u \leq 1/2\right) \geq C,
\end{align*}
and, therefore, 
\begin{align*}
\EE\left(\stau_{h-1} \land \vtau_{a-1} \mid V_0 = a , S_0 - X_0 = h\right) \geq C\cdot \frac{1}{(-4a) \lor g} \ ,
\end{align*}
proving part (vi) of the lemma in the case $h\geq 1$ and $a\leq -\sqrt{g/8}$.

For $a\geq -\sqrt{g/8}$ and $0\leq t \leq  1/(4\sqrt{g/8})$,
\begin{align*}
S_t \geq h-gt^2/2 + at  \geq 
h-gt^2/2 -\sqrt{g/8}t  \geq
h-1/2,
\end{align*}
so the analogous argument gives
\begin{align*}
\EE\left(\stau_{h-1} \land \vtau_{a-1} \mid V_0 = a , H_0 = h\right) \geq C,
\end{align*}
completing the proof of part (vi) of the lemma in the case $h\geq 1$ and $a\in\Reals$.
\end{proof}

\begin{lem}\label{no5.1}
Assume that $V_0=-g, H_0=0$ and let $\{\zeta_k\}_{k \ge 0}$ be the renewal times defined in \eqref{align:zeta}. Then there exists $a_0>0$ and positive constants $C_k$, $k=1,\dots, 6$, such that for $a,r \ge a_0$,
\begin{align}\label{no5.2}
C_1 \frac 1 {a}e^{-a^2 + 2 a(g-1)} \le \Prob&\left(\vtau_{-a} < \zeta_1\right) \le 
C_2  e^{-a^2 + 2 a(g+1)},\\
\label{no5.3}
C_3 \frac 1 {a}e^{-a^2 - 2 a(g+1)} 
\le \Prob&\left(\vtau_a < \zeta_1 \right) 
\le C_4 \frac 1 a e^{-a^2 - 2 a(g-1)},\\
\label{no5.4}
C _5 \frac 1 {\sqrt{r}} e^{-2gr} \le \Prob&\left(\stau_r < \zeta_1\right) \le C_6 e^{-2gr}.
\end{align}
\end{lem}

\begin{proof}

Recall from Lemma \ref{lem:estzeta} that $\EE (\zeta_1 )< \infty$. 
We can apply \eqref{eq:stationarypi} to a set 
$A \subset \Reals\times [0,\infty)$ and use \eqref{eq:density} 
to see that
\begin{align}\label{no8.30}
\Expect\left( \int_{0}^{\zeta_1} \Indicator_A(Z_u) du \right) 
= 
C \pi(A)
= C \iint\limits _A
\frac{2g}{\sqrt{\pi} } e^{-2gh}e^{-(v+g)^2}  dv dh .
\end{align}

(i)  Fix $a\geq g +2$.
We apply \eqref{no8.30} to the set 
$A = (-a-1, -a]\times [0,1]$  
to see that
\begin{align}\label{no7.7}
\Expect&\left( \int_{0}^{\zeta_1} \Indicator(-a-1 \leq V_u\leq -a, H_u\leq 1) du \right) 
= C \int_0^1 \int_{-a-1}^{-a}
\frac{2g}{\sqrt{\pi} } e^{-2gh}e^{-(v+g)^2}  dv dh \\
&\geq C  \int_{-a-1}^{-a}
 e^{-(v+g)^2}  dv \nonumber
\geq C  e^{-(a+1-g)^2}
\geq C'  e^{-a^2 - 2 a(1-g)}.
\end{align}

Let $T=\inf\{t\geq 0: V_t \in [-a-1,-a], H_t \leq 1\}$. Note that  \eqref{no7.7} only involves paths of $(V_t,H_t)$ such that $T<\zeta_1$. It follows from Remark \eqref{oc16.1} that $T+\vtau_{-g}\circ\theta_T = \zeta_1$ on the event $\{T < \zeta_1\}$. Lemma \ref{no5.5} (i) and the strong Markov property applied at $T$ imply that
\begin{align*}
\Expect&\left( \int_{0}^{\zeta_1} \Indicator(-a-1\leq V_u\leq -a, H_u\leq 1) du \right)
\leq
\EE\left(\Indicator(T < \zeta_1)(\zeta_1-T)\right) \\
& \leq \PP\left(T < \zeta_1\right) \times
\sup_{v \in [-a-1,-a], h \in [0,1]}\EE(\vtau_{-g} \mid V_0=v, H_0=h)
\leq C a 
\PP\left(T < \zeta_1\right).
\end{align*}
This and \eqref{no7.7} yield
\begin{align*}
\PP\left(\vtau_{-a} < \zeta_1 \right) 
\geq 
\PP\left(T < \zeta_1\right)
\geq
C \frac 1 {a}e^{-a^2 - 2 a(1-g)}.
\end{align*}
This proves the lower bound in \eqref{no5.2} for $a \geq g+2$.

\medskip

(ii) Fix $a\geq g +2$.
We apply \eqref{no8.30} to the set 
$A = (-\infty, -a+1]\times [0,\infty)$ and use \eqref{oc13.3}
to see that
\begin{align}\label{no5.7}
\Expect&\left( \int_{0}^{\zeta_1} \Indicator(V_u\leq -a+1) du \right) 
= C \int_0^\infty \int^{-a+1}_{-\infty}
\frac{2g}{\sqrt{\pi} } e^{-2gh}e^{-(v+g)^2}  dv dh \nonumber \\
&= C  \int^{-a+1}_{-\infty}
\frac{1}{\sqrt{\pi} } e^{-(v+g)^2}  dv
\leq C' \frac 1 {a-1-g}  e^{-(a-1-g)^2}
\leq C'' \frac 1 a e^{-a^2 + 2 a(g+1)}.
\end{align}

Let $T = \inf\{t\geq \vtau_{-a}: V_t = -a+1\}$. Note that if $\vtau_{-a} < \zeta_1$ then $T< \zeta_1$. Lemma \ref{no5.5} (ii) implies that
\begin{align*}
\Expect\left( \int_{0}^{\zeta_1} \Indicator(V_u\leq -a+1) du \right) 
\geq
\PP\left(\vtau_{-a} < \zeta_1 \right)
\times \inf_{h \ge 0}\EE( \vtau_{-a+1} \mid V_0=a, H_0=h)
\geq \frac C a \PP\left(\vtau_{-a} < \zeta_1 \right),
\end{align*}
so \eqref{no5.7} yields
\begin{align*}
\PP\left(\vtau_{-a} < \zeta_1 \right) \leq C  e^{-a^2 + 2 a(g+1)}.
\end{align*}
This proves the upper bound in \eqref{no5.2} for $a \geq g+2$.

\medskip

(iii) Fix $a\geq 2$.
We apply \eqref{no8.30} to the set 
$A = [a,\infty) \times [0,\infty)$  
to see that
\begin{align}\label{no8.3}
\Expect&\left( \int_{0}^{\zeta_1} \Indicator(V_u\geq a) du \right) 
= C \int_0^\infty \int_a^\infty
\frac{2g}{\sqrt{\pi} } e^{-2gh}e^{-(v+g)^2}  dv dh \nonumber\\
&\geq C  \int_a^\infty
 e^{-(v+g)^2}  dv 
\geq C'  e^{-(a+1+g)^2}
\geq C''  e^{-a^2 - 2 a(g+1)}.
\end{align}

It follows from Remark \eqref{oc16.1} that on the event $\{\vtau_{-g} < \zeta_1\}$, $V$ will not take values greater than $-g$ on the interval $ [ \vtau_{-g}, \zeta_1]$.  Lemma \ref{no5.5} (iii) and the strong Markov property applied at $\vtau_{a}$ imply that
\begin{align*}
\Expect&\left( \int_{0}^{\zeta_1} \Indicator(V_u\geq a) du \right) 
\leq
\PP\left(\vtau_{a} < \zeta_1\right)
\EE(\vtau_{-g} \mid V_0=a, H_0=0)
\leq C a 
\PP\left(\vtau_{a} < \zeta_1\right).
\end{align*}
This and \eqref{no8.3} yield
\begin{align*}
\PP\left(\vtau_{a} < \zeta_1 \right) 
\geq 
C \frac 1 {a}e^{-a^2 - 2 a(g+1)}.
\end{align*}
This proves the lower bound in \eqref{no5.3} for $a \geq 2$.

\medskip

(iv) Fix $a\geq 2$.
We apply \eqref{no8.30} to the set 
$A = [ a-1, \infty)\times [0,\infty)$ and use \eqref{oc13.3}
to see that
\begin{align}\label{no5.6}
\Expect&\left( \int_{0}^{\zeta_1} \Indicator(V_u\geq a-1) du \right) 
= C \int_0^\infty \int_{a-1}^\infty
\frac{2g}{\sqrt{\pi} } e^{-2gh}e^{-(v+g)^2}  dv dh \nonumber \\
&= C  \int_{a-1}^\infty
\frac{1}{\sqrt{\pi} } e^{-(v+g)^2}  dv
\leq C' \frac 1 {a-1+g}  e^{-(a-1+g)^2}
\leq C'' \frac 1 a e^{-a^2 - 2 a(g-1)}.
\end{align}

Let $T = \inf\{t\geq \vtau_a: V_t = a-1\}$. Note that if $\vtau_a < \zeta_1$ then $T< \zeta_1$. Lemma \ref{no5.5} (iv) implies that
\begin{align*}
\Expect\left( \int_{0}^{\zeta_1} \Indicator(V_u\geq a-1) du \right) 
\geq
\PP\left(\vtau_a < \zeta_1 \right)
\EE(\vtau_{a-1} \mid V_0=a, H_0=0)
\geq (1/g)\PP\left(\vtau_a < \zeta_1 \right),
\end{align*}
so \eqref{no5.6} yields
\begin{align*}
\PP\left(\vtau_a < \zeta_1 \right) \leq C \frac 1 a e^{-a^2 - 2 a(g-1)}.
\end{align*}
This proves the upper bound in \eqref{no5.3} for $a \geq 2$.

\medskip

(v) Consider $r\geq 1$. 
We apply \eqref{no8.30} to the set 
$A = [-g-2, -g-1]\times [r,\infty)$  
to see that
\begin{align}\label{no8.12}
\Expect&\left( \int_{0}^{\zeta_1} \Indicator(-g-2 \leq V_u\leq -g-1, H_u\geq r) du \right) 
\nonumber \\
&= C \int_r^\infty \int_{-g-2}^{-g-1}
\frac{2g}{\sqrt{\pi} } e^{-2gh}e^{-(v+g)^2}  dv dh 
\geq C  e^{-2gr}.
\end{align}

Let $T=\inf\{t\geq 0: V_t \in [-g-2, -g-1], H_t \geq r\}$.
It follows from Remark \eqref{oc16.1} that on the event $\{T < \zeta_1\}$, $ \inf\{t\geq T: V_t = -g\}= \zeta_1$.  Lemma \ref{no5.5} (v) and the strong Markov property applied at $T$ imply that
\begin{align*}
\Expect&\left( \int_{0}^{\zeta_1} \Indicator(-g-2 \leq V_u\leq -g-1, H_u\geq r) du \right) \\
& \leq
\PP\left(T < \zeta_1\right)
\times \sup_{v \in [-g-2, -g-1], h \ge r}\EE(\vtau_{-g} \mid V_0=v, H_0=h)
\leq C \sqrt{r}
\PP\left(T < \zeta_1\right).
\end{align*}
This and \eqref{no8.12} yield
\begin{align*}
\PP\left(\stau_{h} < \zeta_1 \right) 
\geq 
\PP\left(T < \zeta_1\right)
\geq
C \frac 1 {\sqrt{r}}e^{-2gr}.
\end{align*}
This proves the lower bound in \eqref{no5.4} for $r \geq 1$.

\medskip

(vi)  Fix $r\geq 2$.
We apply \eqref{no8.30} to the set 
$A = [k-2, k+1)\times [r-1,\infty)$ and use  \eqref{oc13.3}
to see that
\begin{align}\label{no5.8}
\Expect&\left( \int_{0}^{\zeta_1} \Indicator(k-2 \leq V_u< k+1, H_u \geq r-1) du \right)  \\
&= C \int_{r-1}^\infty \int_{k-2}^{k+1}
\frac{2g}{\sqrt{\pi} } e^{-2gh}e^{-(v+g)^2}  dv dh 
= C e^{-2g(r-1)} 
\int_{k-2}^{k+1}\frac{1}{\sqrt{\pi} } e^{-(v+g)^2}  dv \nonumber\\
&\leq C e^{-2g(r-1)}
\frac 1 {|k|+g +1} e^{-((|k|+g-2)\lor 0)^2}.
\nonumber
\end{align}

Let  
\begin{align*}
T_k = \inf\{t\geq \stau_r: H_t = r-1
\text{  or  }
V_t = k-2
\}.
\end{align*}
Note that if $\stau_r < \zeta_1$ then $T_k\leq \zeta_1$ for every integer $k$. Lemma \ref{no5.5} (vi) implies that
\begin{align*}
\Expect&\left( \int_{0}^{\zeta_1} \Indicator(k-2 \leq V_u< k+1, H_u \geq r-1) du \right) \\
&\geq
\PP\left(\stau_r < \zeta_1, k-1\leq V_{\stau_r}< k  \right)
\times \inf_{v \in [k-1,k]}\EE(\stau_{r-1} \land \vtau_{v-1} \mid V_0=v, H_0=r )\\
&\geq C \frac1 {(-k) \lor 1} 
\PP\left(\stau_r < \zeta_1, k-1\leq V_{\stau_r}< k  \right),
\end{align*}
so \eqref{no5.8} yields
\begin{align*}
\PP\left(\stau_r < \zeta_1, k-1\leq V_{\stau_r}< k  \right)
 \leq C ((-k)\lor 1)  e^{-2g(r-1)}
\frac 1 {|k|+g +1} e^{-((|k|+g-2)\lor 0)^2}.
\end{align*}
It follows that
\begin{align*}
\PP&\left(\stau_r < \zeta_1 \right)
= \sum_{k=-\infty}^\infty
\PP\left(\stau_r < \zeta_1, k-1\leq V_{\stau_r}< k  \right)\\
& \leq \sum_{k=-\infty}^\infty C ((-k)\lor 1)  e^{-2g(r-1)}
\frac 1 {|k|+g +1} e^{-((|k|+g-2)\lor 0)^2}
\leq C' e^{-2gr}.
\end{align*}
This proves the upper bound in \eqref{no5.4} for $r \geq 2$.
\end{proof}

\begin{proof}[Proof of Theorem \ref{thm:velfluc}]

Fix arbitrarily small $\eps >0$.
By \eqref{no5.3} of Lemma \ref{no5.1}, for  $n \ge 0$,
$$
\Prob\left(\sup_{t \in [\zeta_n, \zeta_{n+1}]}V_t > \sqrt{(1+\eps) \log n}\right) \le \frac{C}{n^{1+\eps/2}}
$$
and
$$
\Prob\left(\sup_{t \in [\zeta_n, \zeta_{n+1}]}V_t > \sqrt{(1-\eps)\log n}\right) \ge \frac{C}{n^{1-\eps/2}}.
$$
As the $\{\zeta_n\}_{n \ge 0}$ are renewal times, therefore by the Borel Cantelli lemma, a.s.,
\begin{equation}\label{equation:BC}
\sqrt{1-\eps} \le \limsup_{n \rightarrow \infty}\frac{\sup_{t \in [\zeta_n, \zeta_{n+1}]}V_t}{\sqrt{\log n}} \le \sqrt{1+\eps}.
\end{equation}
By Lemma \ref{lem:estzeta}, $\Expect(\zeta_1)< \infty$ and thus, by the Strong Law of Large Numbers, $\zeta_n/n \rightarrow \Expect(\zeta_1)$, a.s. From the lower bound in \eqref{equation:BC}, with probability 1, there exists a subsequence $n_k \rightarrow \infty$ and $t_{n_k} \in [\zeta_{n_k}, \zeta_{n_k+1}]$ such that $V_{t_{n_k}}/\sqrt{\log n_k} \ge \sqrt{1-\eps}$. Moreover, the SLLN implies that, a.s., $\log t_{n_k} \le (1+\eps)\log n_k$ for sufficiently large $k$. Therefore, a.s.,
$$
\frac{V_{t_{n_k}}}{\sqrt{\log t_{n_k}}} \ge \sqrt{\frac{1-\eps}{1+\eps}}
$$
for sufficiently large $k$.
Since this holds for every $\eps >0$, we obtain, a.s.,
\begin{equation}\label{no8.20}
 \limsup_{t \rightarrow \infty}\frac{V_t}{\sqrt{\log t}} \geq 1.
\end{equation}

From the upper bound in \eqref{equation:BC} and the SLLN, we see that, a.s., there is a positive integer $n_0$ such that for all $n \ge n_0$, $V_t/ \sqrt{\log n} \le \sqrt{1+\eps}$ and $\log t \ge (1-\eps)\log n$ for all $t \in [\zeta_n, \zeta_{n+1}]$. These imply
$$
\frac{V_t}{\sqrt{\log t}} \le \sqrt{\frac{1+\eps}{1-\eps}}
$$
for all $t \ge \zeta_{n_0}$. Since $\eps>0$ can be taken arbitrarily small,
we have, a.s.,
\begin{equation*}
 \limsup_{t \rightarrow \infty}\frac{V_t}{\sqrt{\log t}} \leq 1.
\end{equation*}
This inequality and \eqref{no8.20} prove the second equality in \eqref{equation:flucplus}.

The proofs of the first equality in \eqref{equation:flucplus} and the equality in \eqref{equation:gapplus} follow similarly by using \eqref{no5.2} and \eqref{no5.4} of Lemma \ref{no5.1}, respectively. The claim in \eqref{equation:gapminus} follows from the fact that $\zeta_k$'s are i.i.d. with $\EE \zeta_k>0$, and  $S_t-X_t=0$ if
$t$ is a renewal time.
\end{proof}

\begin{proof}[Proof of Theorem \ref{thm:stronglaw}]
We have
\begin{align*}
X_t-(B_t-gt)&=(X_0+V_0) -V_t,\\
S_t-(B_t-gt)&= X_t-(B_t-gt) + (S_t - X_t)
= (X_0+V_0) -V_t + (S_t - X_t) .
\end{align*}
 These identities and Theorem \ref{thm:velfluc}
easily imply the assertions made in Theorem \ref{thm:stronglaw}.
\end{proof}

\noindent\textbf{Acknowledgements}\\
We are grateful to Amarjit Budhiraja and Ruth Williams for very helpful advice.
S.B. gratefully acknowledges support from EPSRC Research Grant EP/K013939. K.B. was partially supported by NSF Grant DMS-1206276. M.D. was supported by Programa Iniciativa Cientifica Milenio grant number NC130062 through the Nucleus Millenium Stochastic Models of Complex and Disordered Systems.
\bibliographystyle{plain}
\bibliography{inertgrav}
\bigskip
\bigskip
\noindent
{\sc Sayan Banerjee}, {\em University of North Carolina at Chapel Hill}, {\tt <sayan@email.unc.edu>} \\
{\sc Krzysztof Burdzy}, {\em University of Washington}, {\tt <burdzy@math.washington.edu>}\\
{\sc Mauricio Duarte}, {\em Universidad Andres Bello}, {\tt <mauricio.duarte@unab.cl>}
\end{document}